\documentclass{amsart}
\usepackage{lahn}
\usepackage[ style = alphabetic ]{biblatex}

\title{Reducible Suspensions of Anosov Representations}
\author{Max Lahn}

\DeclareDocumentCommand{\K}{}{ \mathbb{ K } }

\begin{document}

\begin{abstract}

We study through the lens of Anosov representations the dynamical properties of reducible suspensions of linear representations of non-elementary hyperbolic groups, which are linear representations preserving and acting weakly unipotently on a proper non-zero subspace. We characterize when reducible suspensions are discrete and (almost) faithful, quasi-isometrically embedded, and Anosov. Anosov reducible suspensions correspond to points in bounded convex domains in a finite-dimensional real vector space. Stronger characterizations of such domains for symmetric Anosov representations allow us to find deformations of Borel Anosov representations which retain some but not all of the Anosov conditions and to compute examples of non-Anosov limits of Anosov representations.

\end{abstract}

\maketitle

\section{Introduction} \label{sec:introduction}

Hitchin components consist of deformations of irreducibly embedded Fuchsian representations of closed hyperbolic surface groups. These Hitchin representations are Borel Anosov \cite[Theorem~4.1, Theorem~4.2]{Lab06} and therefore discrete, faithful \cite[Theorem~1.5]{Lab06}, and quasi-isometric embeddings \cite[Theorem~5.3]{GW12}. On the other hand, large deformations of \emph{reducibly} embedded Fuchsian representations of hyperbolic surface groups need not be Anosov, discrete, faithful, or quasi-isometrically embedded; the most relevant examples in $ \SL_{ 3 } \of{ \R } $ are due to Barbot \cite{Bar01,Bar10}.

In this article, we generalize these examples of Barbot to arbitrary word hyperbolic groups and specific matrix groups of higher rank, and we investigate the dynamical properties of certain deformations within the space of reducible representations. Specifically, fix a non-elementary word hyperbolic group $ \Gamma $, a field $ \K $, either $ \R $ or $ \C $, and an integer $ d \geq 3 $. We aim to study the dynamical properties of what we call \emph{reducible suspensions} of linear representations of $ \Gamma $, which are those representations $ \Gamma \to \SL_{ d }^{ * } \of{ \K } $ which preserve and act up to projectivization by weakly unipotent transformations of a proper non-zero subspace of $ \K^{ d } $. Here, $ \SL_{ d }^{ * } \of{ \K } $ is the matrix group
\[
\SL_{ d }^{ * } \of{ \K } = \set{ \vect{ A } \in \GL_{ d } \of{ \K } : \abs{ \det \of{ \vect{ A } } } = 1 } ,
\]
which is a smooth fiber bundle over $ \PGL_{ d } \of{ \K } $ with fiber $ \S^{ 0 } = \set{ \pm 1 } $ if $ \K = \R $ and $ \S^{ 1 } $ if $ \K = \C $.

Concretely, a reducible suspension of a representation $ \zeta \colon \Gamma \to \SL_{ q }^{ * } \of{ \K } $ is conjugate to a representation of the form $ \rho = \rho_{ p , q } \of{ \varphi , \xi , \zeta , \kappa } \colon \Gamma \to \SL_{ d }^{ * } \of{ \K } $ defined by the block upper-triangular formula
\[
\rho \of{ \gamma } = \begin{pmatrix} e^{ \frac{ 1 }{ p } \varphi \of{ \gamma } } \xi \of{ \gamma } & \kappa \of{ \gamma } \\ \vect{ 0 } & e^{ -\frac{ 1 }{ q } \varphi \of{ \gamma } } \zeta \of{ \gamma } \end{pmatrix}
\]
for some decomposition $ d = p + q $ with $ p \geq 1 $ and $ q \geq 2 $, homomorphism $ \varphi \colon \Gamma \to \R $, representation $ \xi \colon \Gamma \to U_{ p }^{ * } \of{ \K } $, and map $ \kappa \colon \Gamma \to \K^{ p \times q } $, whose images are considered as $ p \times q $ matrices with entries in $ \K $. Here, $ U_{ p }^{ * } \of{ \K } $ is the set
\[
U_{ p }^{ * } \of{ \K } = \set{ \vect{ A } \in \SL_{ p }^{ * } \of{ \K } : \lambda_{ 1 } \of{ \vect{ A } } = \dotsb = \lambda_{ p } \of{ \vect{ A } } = 1 }
\]
of \emph{weakly unipotent} matrices, those whose complex eigenvalues are all of unit modulus; this is not a group, and so we clarify that by a representation $ \Gamma \to U_{ p }^{ * } \of{ \K } $ we mean a representation $ \Gamma \to \SL_{ p }^{ * } \of{ \K } $ whose image is contained in $ U_{ p }^{ * } \of{ \K } $. In order for the above formula for $ \rho \of{ \gamma } $ to define a homomorphism, we require that $ \kappa \of{ 1 } = \vect{ 0 } $ and that
\[
\kappa \of{ \gamma \delta } = e^{ \frac{ 1 }{ p } \varphi \of{ \gamma } } \xi \of{ \gamma } \kappa \of{ \delta } + e^{ -\frac{ 1 }{ q } \varphi \of{ \delta } } \kappa \of{ \gamma } \zeta \of{ \delta }
\]
for all $ \gamma , \delta \in \Gamma $. We will describe a map $ \kappa \colon \Gamma \to \K^{ p \times q } $ which satisfies these conditions as being \emph{compatible} with $ \varphi $, $ \xi $, and $ \zeta $.

When $ \Gamma $ is a closed surface group, $ \K = \R $, $ d = 3 $, $ p = 1 $, $ q = 2 $, and $ \zeta $ is discrete and faithful, these representations are precisely the \emph{radial representations} as described by Barbot \cite{Bar01,Bar10}, and whether or not the resulting representation is Anosov is determined entirely by $ \varphi $; specifically, those $ \varphi $ which correspond to Anosov radial representations comprise a metric open ball in the stable norm on $ \Homol^{ 1 } \of{ \Gamma , \R } $ \cite[Theorem~4.2]{Bar10}. Our main results are to extend and refine this result for reducible suspensions of Anosov representations of arbitrary hyperbolic groups; even in this more general setting, whether or not a reducible suspension is Anosov is still determined entirely by the homomorphism $ \varphi \colon \Gamma \to \R $. Specifically, for integers $ p \geq 1 $ and $ 1 \leq k \leq \frac{ q }{ 2 } $ and a representation $ \zeta \colon \Gamma \to \SL_{ q }^{ * } \of{ \K } $, we will denote
\[
A^{ p }_{ k } \of{ \zeta } = \set{ \varphi \in \hom \of{ \Gamma , \R } : \begin{matrix} \textrm{$ \rho_{ p , q } \of{ \varphi , \xi , \zeta , \kappa } $ is $ P_{ k } $-Anosov for some (every) $ \xi \colon \Gamma \to U_{ p }^{ * } \of{ \K } $} \\ \textrm{and some (every) compatible $ \kappa \colon \Gamma \to \K^{ p \times q } $} \end{matrix} } .
\]
That $ A^{ p }_{ k } \of{ \zeta } $ is well-defined is a consequence of \autoref{lem:generality}. In fact, $ A^{ p }_{ k } \of{ \zeta } $ is a bounded convex domain in the finite-dimensional vector space $ \hom \of{ \Gamma , \R } $. 

\begin{restatable}{theorem}{convexity} \label{thm:convexity}

$ A^{ p }_{ k } \of{ \zeta } $ is a bounded, convex, and open in $ \hom \of{ \Gamma , \R } $ for any integers $ p \geq 1 $, $ q \geq 2 $, and $ 1 \leq k \leq \frac{ q }{ 2 } $ and representation $ \zeta \colon \Gamma \to \SL_{ q }^{ * } \of{ \K } $.

\end{restatable}

We note as an immediate consequence of \autoref{prop:eigenvalues} that $ A^{ p }_{ k } \of{ \zeta } $ is empty when $ \zeta $ is not $ P_{ k } $-Anosov, and thus \autoref{thm:convexity} is vacuously true in this case. On the other hand, we note that $ A^{ p }_{ k } \of{ \zeta } $ may be empty even if $ \zeta $ is $ P_{ k } $-Anosov.

Drawing motivation from both Hitchin representations and Barbot's radial representations, our main examples of interest will be reducible suspensions of Anosov $ q $-Fuchsian/Kleinian representations, those representations for which $ \zeta = \iota_{ q } \circ \eta $ factors as the composition of an Anosov representation $ \eta \colon \Gamma \to \SL_{ 2 }^{ * } \of{ \K } $ with the irreducible representation $ \iota_{ q } \colon \SL_{ 2 }^{ * } \of{ \K } \into \SL_{ q }^{ * } \of{ \K } $. Along with Anosov representations with images contained in the indefinite orthogonal/unitary group or the symplectic group, these are examples of \emph{symmetric Anosov} representations, for which we are able to precisely determine which reducible suspensions are Anosov.

\begin{restatable}{theorem}{symmetricAnosov} \label{thm:symmetricAnosov}

For any integers $ p \geq 1 $, $ q \geq 2 $, and $ 1 \leq k \leq \frac{ q }{ 2 } $ and any symmetric $ P_{ k } $-Anosov representation $ \zeta \colon \Gamma \to \SL_{ q }^{ * } \of{ \K } $,
\[
A^{ p }_{ k } \of{ \zeta } = \set{ \varphi \in \hom \of{ \Gamma , \R } : \inf_{ \varphi \of{ \gamma } \neq 0 }{ \frac{ \log \of{ \lambda_{ k } \of{ \zeta \of{ \gamma } } } }{ \abs{ \varphi \of{ \gamma } } } } > \frac{ p + q }{ p q } } .
\]

\end{restatable}

We are also able to gain more easily computable bounds on the sizes of these convex domains in terms of dynamical data associated to the symmetric Anosov representation $ \zeta $. Fix a finite generating set $ \Sigma $ for $ \Gamma $, and note that such a choice defines a uniform norm $ \norm{ \oparg }_{ \infty } $ on $ \hom \of{ \Gamma , \R } $. We denote the metric open ball of radius $ r $ centered at $ 0 $ by $ B \of{ 0 ,  r } $.

\begin{restatable}{corollary}{bounds} \label{cor:bounds}

For any integers $ p \geq 1 $, $ q \geq 2 $, and $ 1 \leq k \leq \frac{ q }{ 2 } $ and any symmetric $ P_{ k } $-Anosov representation $ \zeta \colon \Gamma \to \SL_{ q }^{ * } \of{ \K } $,
\[
\mB \of{ 0 , \frac{ p q }{ p + q } \cdot s_{ k } \of{ \zeta } } \subseteq A^{ p }_{ k } \of{ \zeta } \subseteq \mB \of{ 0 , \frac{ p q }{ p + q } \cdot \max_{ \sigma \in \Sigma }{ \log \of{ \lambda_{ k } \of{ \zeta \of{ \sigma } } } } } .
\]

\end{restatable}

In the above statement, $ s_{ k } \of{ \zeta } $ is the optimal slope for the linear growth of $ \log \of{ \lambda_{ k } \of{ \zeta \of{ \gamma } } } $ in terms of the translation length $ \norm{ \gamma }_{ \Sigma } $; this is expounded upon in \autoref{sec:background}.

If a representation $ \eta \colon \Gamma \to \SL_{ 2 }^{ * } \of{ \K } $ is Anosov, then the corresponding $ q $-Fuchsian/Kleinian representation $ \iota_{ q } \circ \eta \colon \Gamma \to \SL_{ q }^{ * } \of{ \K } $ is symmetric Borel Anosov and therefore satisfies the hypotheses of \autoref{thm:symmetricAnosov} for all integers $ 1 \leq k \leq \frac{ q }{ 2 } $ simultaneously. In this case, a formula for the eigenvalues of $ \iota_{ q } \circ \eta $ (see \autoref{ex:symmetricAnosov}) allows us to compare the convex domains established in \autoref{thm:convexity} and \autoref{thm:symmetricAnosov} for different parabolic subgroups.

\begin{restatable}{corollary}{Fuchsian} \label{cor:Fuchsian}

For any Anosov representation $ \eta \colon \Gamma \to \SL_{ 2 }^{ * } \of{ \K } $,
\[
A^{ p }_{ k } \of{ \iota_{ q } \circ \eta } = \set{ \varphi \in \hom \of{ \Gamma , \R } : \inf_{ \varphi \of{ \gamma } \neq 0 }{ \frac{ \log \of{ \lambda_{ 1 } \of{ \eta \of{ \gamma } } } }{ \abs{ \varphi \of{ \gamma } } } } > \frac{ p + q }{ p q \of{ q - 2 k + 1 } } }
\]
for all integers $ p \geq 1 $, $ q \geq 2 $, and $ 1 \leq k \leq \frac{ q }{ 2 } $. Moreover, $ A^{ p }_{ k } \of{ \iota_{ q } \circ \eta } $ is strictly increasing in $ p $ and $ q $ and strictly decreasing in $ k $; that is,
\begin{align*}
A^{ p }_{ k } \of{ \iota_{ q } \circ \eta } & \subsetneq A^{ p + 1 }_{ k } \of{ \iota_{ q } \circ \eta } & A^{ p }_{ k } \of{ \iota_{ q } \circ \eta } & \subsetneq A^{ p }_{ k } \of{ \iota_{ q + 1 } \circ \eta } & A^{ p }_{ k } \of{ \iota_{ q } \circ \eta } & \supsetneq A^{ p }_{ k + 1 } \of{ \iota_{ q } \circ \eta }
\end{align*}

\end{restatable}

Some of the reducible suspensions of these Anosov $ q $-Fuchsian/Kleinian representations are themselves Anosov, but the boundedness conclusions of \autoref{thm:convexity} and \autoref{cor:bounds} each imply that unlike in the case of Hitchin representations, large deformations of these reducible suspensions need not remain so. In particular, the aforementioned convex domains corresponding to Anosov reducible suspensions are strictly nested, and so we can find deformations of Borel Anosov representations which retain some but not all of the Anosov properties.

\begin{restatable}{corollary}{deformations} \label{cor:deformations}

If $ \Gamma $ admits an Anosov representation in $ \SL_{ 2 }^{ * } \of{ \K } $ and $ d \geq 3 $ is odd, then for all integers $ 1 \leq j \leq \frac{ d }{ 2 } $, there are continuous deformations of a Borel Anosov representation in $ \SL_{ d }^{ * } \of{ \K } $ which are $ P_{ k } $-Anosov for all integers $ 1 \leq k \leq j $ but not $ P_{ k } $-Anosov for any integer $ j < k \leq \frac{ d }{ 2 } $.

\end{restatable}

We will pay particular attention to the reducible suspensions in $ \SL_{ 3 }^{ * } \of{ \K } $ (that is, when $ p = 1 $ and $ q = 2 $), since this setting severely restricts the complexities in the general case. For an example of this, one may note that any reducible representation of $ \Gamma $ in $ \SL_{ 3 }^{ * } \of{ \K } $ is a reducible suspension or its contragredient, and that the nilpotence hypothesis of \autoref{prop:discreteness} are automatically satisfied in this case. From an analysis of these representations, we are able to make conclusions of independent interest.

\begin{restatable}{corollary}{rankTwo} \label{cor:rankTwo}

Any reducible, discrete, and faithful representation of a closed hyperbolic surface group in $ \SL_{ 3 }^{ \pm } \of{ \R } $ or a closed hyperbolic $ 3 $-manifold group in $ \SL_{ 3 }^{ * } \of{ \C } $ is a quasi-isometric embedding.

\end{restatable}

We expect that the conclusions of \autoref{cor:rankTwo} do not hold for irreducible representations.

\subsection*{Acknowledgements} The author gratefully acknowledges the support of NSF grant DMS-1906441, and the personal and professional support of Dick Canary, Teddy Weisman, and Pierre-Louis Blayac, without whom this article would not have been possible.

\section{Background and Notation} \label{sec:background}

Originally introduced by Labourie \cite[Section~3.1.2]{Lab06}, Anosov representations provide a general framework through which many of the desirable properties of convex-cocompact discrete subgroups of rank one Lie groups may be recovered in the more general setting of discrete subgroups of Lie groups of arbitrary rank. Anosov representations encompass the monodromy representations of convex-cocompact real and complex hyperbolic manifolds, Hitchin representations of closed hyperbolic surface groups, maximal representations, and the monodromy representations of closed strictly convex projective manifolds, to name a few examples. 

We recall that throughout this article, $ \Gamma $ is assumed to be a non-elementary word hyperbolic group, and we begin by making the following preliminary definitions. To each finite generating set $ \Sigma $ for $ \Gamma $ is associated a \emph{word length} function $ \abs{ \oparg }_{ \Sigma } \colon \Gamma \to \N $ defined by the formula
\[
\abs{ \gamma }_{ \Sigma } = \inf \set{ n \in \N : \gamma = \sigma_{ 1 } \dotsm \sigma_{ n } \textrm{ for some } \sigma_{ 1 } , \dotsc , \sigma_{ n } \in \Sigma \cup \Sigma^{ -1 } }
\]
and a \emph{translation length} function $ \norm{ \oparg }_{ \Sigma } \colon \Gamma \to \N $ defined by the formula $ \norm{ \gamma }_{ \Sigma } = \inf_{ \delta \in \Gamma }{ \abs{ \delta \gamma \delta^{ -1 } }_{ \Sigma } } $. Moreover, to each matrix $ \vect{ A } \in \SL_{ d }^{ * } \of{ \K } $ we associate $ d $ \emph{eigenvalue magnitudes}
\[
\lambda_{ 1 } \of{ \vect{ A } } \geq \lambda_{ 2 } \of{ \vect{ A } } \geq \dotsb \geq \lambda_{ d } \of{ \vect{ A } } .
\]
These are the magnitudes of the complex eigenvalues of $ \vect{ A } $, counted with algebraic multiplicity and listed in non-increasing order. Finally, given two functions $ f , g \colon \Gamma \to \R $, we will say that $ f \of{ \gamma } $ \emph{grows at least linearly} in $ g \of{ \gamma } $ if there are constants $ a > 0 $ and $ b \geq 0 $ so that
\[
f \of{ \gamma } \geq a g \of{ \gamma } - b
\]
for all $ \gamma \in \Gamma $. With this terminology and the above notation, we can define Anosov representations as follows:

\begin{definition+}

For an integer $ 1 \leq k \leq \frac{ d }{ 2 } $, a representation $ \rho \colon \Gamma \to \GL_{ d } \of{ \K } $ is called \emph{$ P_{ k } $-Anosov} if the $ k $th logarithmic eigenvalue gap $ \log \of{ \frac{ \lambda_{ k } \of{ \rho \of{ \gamma } } }{ \lambda_{ k + 1 } \of{ \rho \of{ \gamma } } } } $ grows at least linearly in the translation length $ \norm{ \gamma }_{ \Sigma } $; that is, $ \rho $ is $ P_{ k } $-Anosov if there are constants $ a > 0 $ and $ b \geq 0 $ so that
\[
\log \of{ \frac{ \lambda_{ k } \of{ \rho \of{ \gamma } } }{ \lambda_{ k + 1 } \of{ \rho \of{ \gamma } } } } \geq a \norm{ \gamma }_{ \Sigma } - b
\]
for all $ \gamma \in \Gamma $. A representation $ \rho \colon \Gamma \to \GL_{ d } \of{ \K } $ is called \emph{Anosov} if it is $ P_{ k } $-Anosov for some integer $ 1 \leq k \leq \frac{ d }{ 2 } $, and $ \rho $ is called \emph{Borel Anosov} if it is $ P_{ k } $-Anosov for all integers $ 1 \leq k \leq \frac{ d }{ 2 } $.

\end{definition+}

We note that a different choice $ \Sigma' $ of finite generating set induces different translation lengths $ \norm{ \oparg }_{ \Sigma' } $ but not a different notion of $ P_{ k } $-Anosov, since each of $ \norm{ \gamma }_{ \Sigma } $ and $ \norm{ \gamma }_{ \Sigma' } $ grows at least linearly in the other.

\begin{remark+}

In the setting of representations of $ \Gamma $ in the reductive Lie group $ \GL_{ d } \of{ \K } $, the above characterization of Anosov representations is due to Kassel--Potrie \cite[Corollary~4.6]{KP22}, and it is equivalent both to the original definition given by Labourie \cite[Section~3.1.2]{Lab06} in the case of closed hyperbolic surface groups and to several subsequent characterizations by Guichard--Wienhard \cite[Definition~2.10, Proposition~3.16]{GW12}, Kapovich--Leeb--Porti \cite[Equivalence~Theorem~1.1]{KLP17} \cite[Theorem~1.5, Corollary~1.6]{KLP18}, Gu\'{e}ritaud--Guichard--Kassel--Wienhard \cite[Theorem~1.3, Theorem~1.7]{Gue17}, Bochi--Potrie--Sambarino \cite[Proposition~4.5, Proposition~4.9]{BPS19}, and Tsouvalas \cite[Theorem~1.1]{Tso22}.

\end{remark+}

A $ P_{ k } $-Anosov representation $ \rho \colon \Gamma \to \GL_{ d } \of{ \K } $ comes equipped with an \emph{Anosov limit map} $ \xi_{ \rho }^{ k } \colon \partial \Gamma \to \Gr_{ k } \of{ \K^{ d } } $ on the Gromov boundary \cite[Definition~2.10]{GW12}. This continuous and $ \rho $-equivariant map is \emph{dynamics-preserving}; that is, for any infinite order $ \gamma \in \Gamma $, $ \xi_{ \rho }^{ k } \of{ \gamma^{ + } } $ is the unique attracting fixed point $ \rho \of{ \gamma }^{ + }_{ k } $ for the action of $ \rho \of{ \gamma } $ on the Grassmannian $ \Gr_{ k } \of{ \K^{ d } } $ \cite[Lemma~3.1]{GW12}.

The $ P_{ k } $-Anosov condition is also well-known for being \emph{stable}; that is, any small deformation of a $ P_{ k } $-Anosov representation is also $ P_{ k } $-Anosov.

\begin{theorem+}[{\cite[Theorem~1.2]{GW12}}] \label{thm:stability}

The set of $ P_{ k } $-Anosov representations $ \Gamma \to \GL_{ d } \of{ \K } $ is open in the representation variety $ \hom \of{ \Gamma , \GL_{ d } \of{ \K } } $.

\end{theorem+}

We will call a representation $ \rho \colon \Gamma \to \GL_{ d } \of{ \K } $ \emph{symmetric $ P_{ k } $-Anosov} if it is $ P_{ k } $-Anosov and $ \lambda_{ k } \of{ \rho \of{ \gamma } } = \lambda_{ k } \of{ \rho \of{ \gamma^{ -1 } } } $ for all $ \gamma \in \Gamma $. Note that this automatically implies also that
\[
\lambda_{ d - k + 1 } \of{ \rho \of{ \gamma } } = \frac{ 1 }{ \lambda_{ k } \of{ \rho \of{ \gamma^{ -1 } } } } = \frac{ 1 }{ \lambda_{ k } \of{ \rho \of{ \gamma } } } = \lambda_{ d - k + 1 } \of{ \rho \of{ \gamma^{ -1 } } }
\]
for all $ \gamma \in \Gamma $.

\begin{examples+} \label{ex:symmetricAnosov}

The following are examples of symmetric Anosov representations:
\begin{itemize}

\item

Let $ \eta \colon \Gamma \to \SL_{ 2 }^{ * } \of{ \K } $ be a Fuchsian/Kleinian representation. The composition $ \iota_{ q } \circ \eta \colon \Gamma \to \SL_{ q }^{ * } \of{ \K } $ with an irreducible representation $ \iota_{ q } \colon \SL_{ 2 }^{ * } \of{ \K } \into \SL_{ q }^{ * } \of{ \K } $ is called \emph{$ q $-Fuchsian} or \emph{$ q $-Kleinian} (depending on whether $ \K = \R $ or $ \K = \C $). If $ \eta $ is Anosov, then $ \iota_{ q } \circ \eta $ is symmetric Borel Anosov, since
\[
\lambda_{ k } \of{ \iota_{ q } \of{ \eta \of{ \gamma } } } = \lambda_{ 1 } \of{ \eta \of{ \gamma } }^{ q - 2 k + 1 }
\]
for all integers $ 1 \leq k \leq q $.

\item 

Fix a matrix $ \Omega \in \GL_{ q } \of{ \K } $, and consider the matrix group
\[
G_{ \Omega } \of{ \K } = \set{ \vect{ A } \in \GL_{ q } \of{ \K } : \Omega \vect{ A } \Omega^{ -1 } = \contra{ \vect{ A } } } .
\]
Then any $ P_{ k } $-Anosov representation $ \zeta \colon \Gamma \to \GL_{ q } \of{ \K } $ with image contained in $ G_{ \Omega } \of{ \K } $ is symmetric $ P_{ k } $-Anosov, since
\[
\lambda_{ k } \of{ \vect{ A } } = \lambda_{ k } \of{ \trans{ \vect{ A } } } = \frac{ 1 }{ \lambda_{ q - k } \of{ \contra{ \vect{ A } } } } = \frac{ 1 }{ \lambda_{ q - k } \of{ \Omega \vect{ A } \Omega^{ -1 } } } = \frac{ 1 }{ \lambda_{ q - k } \of{ \vect{ A } } } = \lambda_{ k } \of{ \vect{ A }^{ -1 } }
\]
for all $ \vect{ A } \in G_{ \Omega } \of{ \K } $. Particular choices of $ \Omega $ imply that an Anosov representation with image contained in the indefinite orthogonal group $ \O \of{ r , s } $, the indefinite unitary group $ \U \of{ r , s } $, or the symplectic group $ \Sp_{ q } \of{ \K } $ is symmetric Anosov.

We also note that this class of examples subsumes the previous, since by work of McGarraghy \cite{McG05}, $ q $-Fuchsian/Kleinian representations have image contained in $ \Sp_{ q } \of{ \K } $ when $ q $ is even and $ \SU \of{ \frac{ q - 1 }{ 2 } , \frac{ q + 1 }{ 2 } } $ when $ q $ is odd.

\end{itemize}

\end{examples+}

\begin{lemma+} \label{lem:symmetricAnosov}

If a representation $ \rho \colon \Gamma \to \GL_{ d } \of{ \K } $ is symmetric $ P_{ k } $-Anosov, then $ \log \of{ \lambda_{ k } \of{ \rho \of{ \gamma } } } $ grows at least linearly in $ \norm{ \gamma }_{ \Sigma } $.

\begin{proof}

Note that
\[
\log \of{ \lambda_{ k } \of{ \rho \of{ \gamma } } } = \frac{ 1 }{ 2 } \log \of{ \lambda_{ k } \of{ \rho \of{ \gamma } }^{ 2 } } = \frac{ 1 }{ 2 } \log \of{ \frac{ \lambda_{ k } \of{ \rho \of{ \gamma } } }{ \lambda_{ d - k + 1 } \of{ \rho \of{ \gamma } } } } \geq \frac{ 1 }{ 2 } \log \of{ \frac{ \lambda_{ k } \of{ \rho \of{ \gamma } } }{ \lambda_{ k + 1 } \of{ \rho \of{ \gamma } } } } 
\]
for all $ \gamma \in \Gamma $. Since $ \rho $ is $ P_{ k } $-Anosov, the last of these expressions grows at least linearly in $ \norm{ \gamma }_{ \Sigma } $. \qedhere

\end{proof}

\end{lemma+}

For such a symmetric $ P_{ k } $-Anosov representation $ \rho \colon \Gamma \to \GL_{ d } \of{ \K } $, we will denote the optimal slope for the linear growth of $ \log \of{ \lambda_{ k } \of{ \rho \of{ \gamma } } } $ in terms of $ \norm{ \gamma }_{ \Sigma } $ by $ s_{ k } \of{ \rho } $; that is,
\[
s_{ k } \of{ \rho } = \sup \set{ a \in \opint{ 0 }{ \infty } : \textrm{There is some } b \geq 0 \textrm{ so that } \log \of{ \lambda_{ k } \of{ \rho \of{ \gamma } } } \geq a \norm{ \gamma }_{ \Sigma } - b \textrm{ for all } \gamma \in \Gamma } .
\]

Finally, we end this section with a useful lemma about commutators in non-elementary hyperbolic groups.

\begin{lemma+} \label{lem:commutators}

Fix a sequence $ \of{ \gamma_{ n } }_{ n = 0 }^{ \infty } $ in a non-elementary hyperbolic group $ \Gamma $. If $ \gamma_{ n } $ is not eventually central, then there is some $ \delta \in \Gamma $ so that $ \comm{ \gamma_{ n } }{ \delta } $ is not eventually $ 1 $.

\begin{proof}

Note that we may without loss of generality pass to subsequences. Since $ \Gamma $ acts on its Gromov boundary $ \partial \Gamma $ as a convergence group, we may pass to a subsequence so that either $ \of{ \gamma_{ n } }_{ n = 0 }^{ \infty } $ is eventually constant or there are ideal points $ a , b \in \partial \Gamma $ so that $ \of{ \gamma_{ n } }_{ n = 0 }^{ \infty } $ converges uniformly to $ a $ on compact subsets of $ \partial \Gamma \setminus \set{ b } $.

First, suppose that $ \of{ \gamma_{ n } }_{ n = 0 }^{ \infty } $ is eventually constant; that is, there is some $ \gamma \in \Gamma $ so that $ \gamma_{ n } = \gamma $ for sufficiently large $ n \in \N $. Since $ \of{ \gamma_{ n } }_{ n = 0 }^{ \infty } $ is not eventually central, $ \gamma $ is not central, and thus $ \comm{ \gamma }{ \delta } \neq 1 $ for some $ \delta \in \Gamma $. In particular, $ \comm{ \gamma_{ n } }{ \delta } $ is eventually not $ 1 $.

On the other hand, suppose that there are ideal points $ a , b \in \partial \Gamma $ so that $ \of{ \gamma_{ n } }_{ n = 0 }^{ \infty } $ converges uniformly to $ a $ on compact subsets of $ \partial \Gamma \setminus \set{ b } $. Since $ \Gamma $ is non-elementary, there is an infinite order $ \delta \in \Gamma $ so that $ \set{ a , b } \cap \set{ \delta^{ + } , \delta^{ - } } = \emptyset $. Suppose for a contradiction that $ \comm{ \gamma_{ n } }{ \delta } = 1 $ for infinitely many $ n \in \N $. Then we may pass to yet another subsequence so that $ \comm{ \gamma_{ n } }{ \delta } = 1 $ for all $ n \in \N $, which implies that
\[
a = \lim_{ n \to \infty }{ \gamma_{ n } \cdot \delta^{ + } } = \lim_{ n \to \infty }{ \of{ \gamma_{ n } \delta \gamma_{ n }^{ -1 } }^{ + } } = \lim_{ n \to \infty }{ \delta^{ + } } = \delta^{ + } \neq a .
\]
This contradiction implies that $ \comm{ \gamma_{ n } }{ \delta } $ is eventually not $ 1 $. \qedhere

\end{proof}

\end{lemma+}

\section{Discreteness and Quasi-Isometric Embeddings} \label{sec:discreteness}

In this section, we will investigate when a reducible suspension is discrete and faithful and when it is a quasi-isometric embedding. To that end, fix for this section a homomorphism $ \varphi \colon \Gamma \to \R $, representations $ \xi \colon \Gamma \to U_{ p }^{ * } \of{ \K } $ and $ \zeta \colon \Gamma \to \SL_{ q }^{ * } \of{ \K } $, and a compatible map $ \kappa \colon \Gamma \to \K^{ p \times q } $, and consider the reducible suspension $ \rho = \rho_{ p , q } \of{ \varphi , \xi , \zeta , \kappa } $. We first aim to characterize when $ \rho $ is discrete and faithful in terms of when $ \zeta $ is.

\begin{proposition+} \label{prop:discreteness}

\begin{enumerate}

\item[1.]

If $ \zeta $ is discrete and faithful (resp. has finite kernel), then $ \rho $ is discrete and faithful (resp. has finite kernel);

\item[2.]

If $ \Gamma $ is centerless, $ \xi \of{ \Gamma } $ is nilpotent, and $ \rho $ is discrete and faithful, then $ \zeta $ is discrete and faithful; and

\item[3.]

If $ \xi \of{ \Gamma } $ is nilpotent and $ \rho $ is discrete and has finite kernel, then $ \zeta $ is discrete and has finite kernel.

\end{enumerate}

\begin{proof}

\begin{enumerate}

\item[1.]

Consider a sequence $ \of{ \gamma_{ n } }_{ n = 0 }^{ \infty } $ in $ \Gamma $ so that $ \lim_{ n \to \infty }{ \rho \of{ \gamma_{ n } } } = \vect{ I }_{ p + q } $. The above formula for $ \rho $ implies that $ \lim_{ n \to \infty }{ e^{ -\frac{ 1 }{ q } \varphi \of{ \gamma_{ n } } } \zeta \of{ \gamma_{ n } } } = \vect{ I }_{ q } $, so that
\begin{align*}
\lim_{ n \to \infty }{ e^{ \frac{ 1 }{ q } \varphi \of{ \gamma_{ n } } } } & = \of{ \lim_{ n \to \infty }{ e^{ -\varphi \of{ \gamma_{ n } } } } }^{ -\frac{ 1 }{ q } } = \of{ \lim_{ n \to \infty }{ \abs{ \det \of{ e^{ -\frac{ 1 }{ q } \varphi \of{ \gamma_{ n } } } \zeta \of{ \gamma_{ n } } } } } }^{ -\frac{ 1 }{ q } } \\
& = \abs{ \det \of{ \lim_{ n \to \infty }{ e^{ -\frac{ 1 }{ q } \varphi \of{ \gamma_{ n } } } \zeta \of{ \gamma_{ n } } } } }^{ -\frac{ 1 }{ q } } = \abs{ \det \of{ \vect{ I }_{ q } } }^{ -\frac{ 1 }{ q } } = 1 .
\end{align*}
In particular,
\[
\lim_{ n \to \infty }{ \zeta \of{ \gamma_{ n } } } = \of{ \lim_{ n \to \infty }{ e^{ \frac{ 1 }{ q } \varphi \of{ \gamma_{ n } } } } } \of{ \lim_{ n \to \infty }{ e^{ -\frac{ 1 }{ q } \varphi \of{ \gamma_{ n } } } \zeta \of{ \gamma_{ n } } } } = 1 \cdot \vect{ I }_{ q } = \vect{ I }_{ q } .
\]
If $ \zeta $ is discrete, this implies that $ \gamma_{ n } \in \ker \of{ \zeta } $ for all but finitely many $ n \in \N $. If $ \ker \of{ \zeta } $ is finite, this further implies that $ \rho $ is discrete and $ \ker \of{ \rho } \subseteq \ker \of{ \zeta } $ is finite. The analogous result for discrete and faithful representations follows by taking $ \ker \of{ \zeta } = \set{ 1 } $.

\item[2.]

For each $ n \in \N $, let $ \vect{ A }_{ n } = e^{ -\frac{ n }{ p } } \vect{ I }_{ p } \oplus e^{ \frac{ n }{ q } } \vect{ I }_{ q } $, and consider the representation $ \rho_{ n } \colon \Gamma \to \SL_{ p + q }^{ * } \of{ \K } $ defined by the formula
\[
\rho_{ n } \of{ \gamma } = \vect{ A }_{ n } \rho \of{ \gamma } \vect{ A }_{ n } = \begin{pmatrix} e^{ \frac{ 1 }{ p } \varphi \of{ \gamma } } \xi \of{ \gamma } & e^{ -\of{ \frac{ p + q }{ p q } } n } \kappa \of{ \gamma } \\ \vect{ 0 } & e^{ -\frac{ 1 }{ q } \varphi \of{ \gamma } } \zeta \of{ \gamma } \end{pmatrix} .
\]
Then the sequence $ \of{ \rho_{ n } }_{ n = 0 }^{ \infty } $ converges pointwise to the representation $ \rho' \colon \Gamma \to \SL_{ p + q }^{ * } \of{ \K } $ defined by the formula
\[
\rho' \of{ \gamma } = e^{ \frac{ 1 }{ p } \varphi \of{ \gamma } } \xi \of{ \gamma } \oplus e^{ -\frac{ 1 }{ q } \varphi \of{ \gamma } } \zeta \of{ \gamma } = \begin{pmatrix} e^{ \frac{ 1 }{ p } \varphi \of{ \gamma } } \xi \of{ \gamma } & \vect{ 0 } \\ \vect{ 0 } & e^{ -\frac{ 1 }{ q } \varphi \of{ \gamma } } \zeta \of{ \gamma } \end{pmatrix} .
\]
Note that $ \rho' = \rho_{ p , q } \of{ \varphi , \xi , \zeta , \vect{ 0 } } $; that is, $ \rho' = \rho $ in the case that $ \kappa \of{ \gamma } = \vect{ 0 } $ for all $ \gamma \in \Gamma $. Furthermore, since each $ \rho_{ n } $ is conjugate to $ \rho = \rho_{ 0 } $, it is discrete and faithful, and given that $ \Gamma $ is a non-elementary word hyperbolic group, the limit $ \rho' $ is discrete and faithful by a well-known consequence \cite[Theorem~8.4]{Kap10} of the Margulis--Zassenhaus lemma \cite{Zas37} \cite[Theorem~1]{KM68}.

In the following computation, for any group $ G $ and integer $ k \in \N $, we will denote by $ c_{ k }^{ G } \colon G^{ k + 1 } \to G $  the $ k $th nested commutator map, which is defined recursively by the formulas $ c_{ 0 }^{ G } \of{ a_{ 0 } } = a_{ 0 } $ and
\[
c_{ k + 1 }^{ G } \of{ a_{ 0 } , \dotsc , a_{ k + 1 } } = \comm{ c_{ k }^{ G } \of{ a_{ 0 } , \dotsc , a_{ k } } }{ a_{ k + 1 } } = c_{ k }^{ G } \of{ a_{ 0 } , \dotsc , a_{ k } } a_{ k + 1 } c_{ k }^{ G } \of{ a_{ 0 } , \dotsc , a_{ k } }^{ -1 } a_{ k + 1 }^{ -1 } .
\]
Let $ \ell $ denote the length of the lower central series for $ \xi \of{ \Gamma } $, so that $ c_{ \ell }^{ \xi \of{ \Gamma } } \of{ \xi \of{ \Gamma }^{ \ell + 1 } } = \set{ \vect{ I }_{ p } } $ is trivial, and suppose for a contradiction that $ \zeta $ is not discrete and faithful. Then
\[
\lim_{ n \to \infty }{ \zeta \of{ \gamma_{ n } } } = \vect{ I }_{ q }
\]
for some sequence $ \of{ \gamma_{ n } }_{ n = 0 }^{ \infty } $ in $ \Gamma $ which is not eventually $ 1 $. Applying \autoref{lem:commutators} $ \ell $ times and passing to a subsequence (which we also denote $ \of{ \gamma_{ n } }_{ n = 0 }^{ \infty } $), we obtain $ \delta_{ 1 } , \dotsc , \delta_{ \ell } \in \Gamma $ so that $ c^{ \Gamma }_{ \ell } \of{ \gamma_{ n } , \delta_{ 1 } , \dotsc , \delta_{ \ell } } $ is not eventually $ 1 $. However, note that
\begin{align*}
\lim_{ n \to \infty }{ \zeta \of{ c_{ \ell }^{ \Gamma } \of{ \gamma_{ n } , \delta_{ 1 } , \dotsc , \delta_{ \ell } } } } & = \lim_{ n \to \infty }{ c^{ \SL_{ q }^{ * } \of{ \K } }_{ \ell } \of{ \zeta \of{ \gamma_{ n } } , \zeta \of{ \delta_{ 1 } } , \dotsc , \zeta \of{ \delta_{ \ell } } } } \\
& = c^{ \SL_{ q }^{ * } \of{ \K } }_{ \ell } \of{ \lim_{ n \to \infty }{ \zeta \of{ \gamma_{ n } } } , \zeta \of{ \delta_{ 1 } } , \dotsc , \zeta \of{ \delta_{ \ell } } } \\
& = c_{ \ell }^{ \SL_{ q }^{ * } \of{ \K } } \of{ \vect{ I }_{ q } , \zeta \of{ \delta_{ 1 } } , \dotsc , \zeta \of{ \delta_{ \ell } } } = \vect{ I }_{ q } ,
\end{align*}
and so
\[
\lim_{ n \to \infty }{ \rho' \of{ c^{ \Gamma }_{ \ell } \of{ \gamma_{ n } , \delta_{ 1 } , \dotsc , \delta_{ \ell } } } } = \lim_{ n \to \infty }{ \begin{pmatrix} \vect{ I }_{ p } & \vect{ 0 } \\ \vect{ 0 } & \zeta \of{ c^{ \Gamma }_{ \ell } \of{ \gamma_{ n } , \delta_{ 1 } , \dotsc , \delta_{ \ell } } } \end{pmatrix} } = \begin{pmatrix} \vect{ I }_{ p } & \vect{ 0 } \\ \vect{ 0 } & \vect{ I }_{ q } \end{pmatrix} = \vect{ I }_{ p + q } .
\]
This contradiction implies that $ \zeta $ is discrete and faithful.

\item[3.]

We will first show that $ \ker \of{ \rho } \subseteq \ker \of{ \varphi } \cap \ker \of{ \xi } \cap \ker \of{ \zeta } \cap \kappa^{ -1 } \of{ \set{ \vect{ 0 } } } $. To that end, let $ \gamma \in \ker \of{ \rho } $, and note that
\begin{align*}
\varphi \of{ \gamma } & = \log \of{ e^{ \varphi \of{ \gamma } } } = \log \of{ e^{ \varphi \of{ \gamma } } \abs{ \det \of{ \xi \of{ \gamma } } } } = \log \of{ \abs{ \det \of{ e^{ \frac{ 1 }{ p } \varphi \of{ \gamma } } \xi \of{ \gamma } } } } = \log \of{ \abs{ \det \of{ \vect{ I }_{ p } } } } = 0 ,
\end{align*}
so that
\[
\begin{pmatrix} \xi \of{ \gamma } & \kappa \of{ \gamma } \\ \vect{ 0 } & \zeta \of{ \gamma } \end{pmatrix} = \begin{pmatrix} e^{ \frac{ 1 }{ p } \varphi \of{ \gamma } } \xi \of{ \gamma } & \kappa \of{ \gamma } \\ \vect{ 0 } & e^{ -\frac{ 1 }{ q } \varphi \of{ \gamma } } \zeta \of{ \gamma } \end{pmatrix} = \vect{ I }_{ p + q } = \begin{pmatrix} \vect{ I }_{ p } & \vect{ 0 } \\ \vect{ 0 } & \vect{ I }_{ q } \end{pmatrix} .
\]
In summary, we have shown that if $ \rho \of{ \gamma } = \vect{ I }_{ p + q } $, then $ \varphi \of{ \gamma } = 0 $, $ \xi \of{ \gamma } = \vect{ I }_{ p } $, $ \zeta \of{ \gamma } = \vect{ I }_{ q } $, and $ \kappa \of{ \gamma } = \vect{ 0 } $. Thus these maps descend to a homomorphism $ \overline{ \varphi } \colon \overline{ \Gamma } \to \R $, representations $ \overline{ \rho } \colon \overline{ \Gamma } \to \SL_{ p + q }^{ * } \of{ \K } $, $ \overline{ \xi } \colon \overline{ \Gamma } \to U_{ p }^{ * } \of{ \K } $, and $ \overline{ \zeta } \colon \overline{ \Gamma } \to \SL_{ q }^{ * } \of{ \K } $, and a compatible $ \overline{ \kappa } \colon \overline{ \Gamma } \to \K^{ p \times q } $, where $ \overline{ \Gamma } $ is the non-elementary hyperbolic group $ \faktor{ \Gamma }{ \ker \of{ \rho } } $. Note that $ \overline{ \rho } $ is discrete and faithful and itself a reducible suspension (of $ \overline{ \zeta } $).

As a finitely generated linear group, $ \overline{ \Gamma } $ contains a finite index torsion-free subgroup $ \overline{ \Delta } $ by the Selberg lemma \cite[Lemma~8]{Sel62}. Note that $ \overline{ \Delta } $ is itself a torsion-free non-elementary hyperbolic group and therefore centerless. Applying part (2) to $ \rest{ \overline{ \rho } }_{ \overline{ \Delta } } = \rho_{ p , q } \of{ \rest{ \overline{ \varphi } }_{ \overline{ \Delta } } , \rest{ \overline{ \xi } }_{ \overline{ \Delta } } , \rest{ \overline{ \zeta } }_{ \overline{ \Delta } } , \rest{ \overline{ \kappa } }_{ \overline{ \Delta } } } $, we find that $ \rest{ \overline{ \zeta } }_{ \overline{ \Delta } } $ is discrete and faithful. Let $ \Delta $ be the preimage of $ \overline{ \Delta } $ under the quotient map $ \Gamma \onto \overline{ \Gamma } $, and note that $ \rest{ \zeta }_{ \Delta } $ is discrete and has finite kernel, since $ \ker \of{ \rho } $ is finite. Since $ \Delta $ has finite index $ \ind{ \Gamma }{ \Delta } = \ind{ \overline{ \Gamma } }{ \overline{ \Delta } } $ in $ \Gamma $, this implies that $ \zeta $ is discrete and has finite kernel. \qedhere

\end{enumerate}

\end{proof}

\end{proposition+}

\begin{remark+}

The hypothesis that $ \Gamma $ is centerless in (2) of the above result is necessary; for example, extend a discrete and faithful representation $ \zeta \colon \Gamma \to \SL_{ q }^{ * } \of{ \K } $ to a representation $ \widehat{ \zeta } \colon \Gamma \times \faktor{ \Z }{ 2 \Z } \to \SL_{ q }^{ * } \of{ \K } $ with the formula $ \widehat{ \zeta } \of{ \gamma , k + 2 \Z } = \zeta \of{ \gamma } $. Then $ \widehat{ \zeta } \of{ \gamma } $ is not faithful, but the reducible suspension $ \rho \colon \Gamma \times \faktor{ \Z }{ 2 \Z } \to \SL_{ d }^{ * } \of{ \K } $ defined by the formula $ \rho \of{ \gamma , k + 2 \Z } = \of{ -1 }^{ k } \vect{ I }_{ p } \oplus \widehat{ \zeta } \of{ \gamma } $ is discrete and faithful.

\end{remark+}

As a consequence of \autoref{prop:discreteness}, we obtain hypotheses under which a reducible suspension is discrete and faithful (resp. has finite kernel) if and only if the original representation is.

\begin{corollary+} \label{cor:discreteness}

Suppose that $ \xi $ has nilpotent image.

\begin{enumerate}

\item[1.]

$ \rho $ is discrete and has finite kernel if and only if $ \zeta $ is discrete and has finite kernel; and

\item[2.]

If $ \Gamma $ is centerless, then $ \rho $ is discrete and faithful if and only if $ \zeta $ is discrete and faithful.

\end{enumerate}

\begin{proof}

(1) follows from (1) and (3) of \autoref{prop:discreteness}, and (2) follows from (1) and (2) of \autoref{prop:discreteness}. \qedhere

\end{proof}

\end{corollary+}

We now consider when $ \rho $ is a quasi-isometric embedding in terms of when $ \zeta $ is a quasi-isometric embedding. We recall that work of Delzant--Guichard--Labourie--Mozes \cite{DGLM11} implies that these conditions are equivalent to inducing displacing actions on the appropriate symmetric spaces.

\begin{proposition+} \label{prop:displacement}

If $ \zeta $ is a quasi-isometric embedding, then $ \rho $ is a quasi-isometric embedding.

\begin{proof}

Denote by $ \len_{ \zeta } $ and $ \len_{ \rho } $ the translation lengths of the respective actions of $ \Gamma $ on the symmetric spaces of $ \SL_{ q }^{ * } \of{ \K } $ and $ \SL_{ p + q }^{ * } \of{ \K } $, respectively. By \cite[Corollary~4.0.6]{DGLM11}, the first of these actions is \emph{displacing}; that is, $ \len_{ \zeta } $ grows at least linearly in the translation length $ \norm{ \oparg }_{ \Sigma } $. Note that
\begin{align*}
\len_{ \rho } \of{ \gamma } & = \sqrt{ \abs{ \log \of{ \lambda_{ 1 } \of{ \rho \of{ \gamma } } } }^{ 2 } + \dotsb + \abs{ \log \of{ \lambda_{ p + q } \of{ \rho \of{ \gamma } } } }^{ 2 } } \geq \frac{ \log \of{ \lambda_{ 1 } \of{ \rho \of{ \gamma } } } - \log \of{ \lambda_{ p + q } \of{ \rho \of{ \gamma } } } }{ 2 } \\
& = \frac{ 1 }{ 2 } \log \of{ \frac{ \lambda_{ 1 } \of{ \rho \of{ \gamma } } }{ \lambda_{ p + q } \of{ \rho \of{ \gamma } } } } \geq \frac{ 1 }{ 2 } \log \of{ \frac{ e^{ -\frac{ 1 }{ q } \varphi \of{ \gamma } } \lambda_{ 1 } \of{ \zeta \of{ \gamma } } }{ e^{ -\frac{ 1 }{ q } \varphi \of{ \gamma } } \lambda_{ q } \of{ \zeta \of{ \gamma } } } } = \frac{ 1 }{ 2 } \log \of{ \frac{ \lambda_{ 1 } \of{ \zeta \of{ \gamma } } }{ \lambda_{ q } \of{ \zeta \of{ \gamma } } } } \\
& \geq \frac{ 1 }{ 2 } \max_{ 1 \leq k \leq q }{ \abs{ \log \of{ \lambda_{ k } \of{ \zeta \of{ \gamma } } } } } \geq \frac{ \sqrt{ \abs{ \log \of{ \lambda_{ 1 } \of{ \zeta \of{ \gamma } } } }^{ 2 } + \dotsb + \abs{ \log \of{ \lambda_{ q } \of{ \zeta \of{ \gamma } } } }^{ 2 } } }{ 2 \sqrt{ q } } = \frac{ \len_{ \zeta } \of{ \gamma } }{ 2 \sqrt{ q } }
\end{align*}
for all $ \gamma \in \Gamma $. Thus $ \len_{ \rho } $ grows at least linearly in $ \norm{ \oparg }_{ \Sigma } $, and so the induced action of $ \Gamma $ on the symmetric space of $ \SL_{ p + q }^{ * } \of{ \K } $ is displacing. By \cite[Proposition~2.2.1, Lemma~2.0.1]{DGLM11} $ \rho $ is a quasi-isometric embedding. \qedhere

\end{proof}

\end{proposition+}

Finally, we will apply the previous results to representations in $ \SL_{ 3 }^{ * } \of{ \K } $; that is, we will consider the case $ p = 1 $ and $ q = 2 $. In this case, the fact that $ U_{ 1 }^{ * } \of{ \K } = \SL_{ 1 }^{ * } \of{ \K } $ is abelian will imply that some of the hypotheses in the previous and upcoming results are satisfied automatically. For example, any representation $ \xi \colon \Gamma \to U_{ 1 }^{ * } \of{ \R } $ has abelian and therefore nilpotent image. Together with \autoref{cor:discreteness}, this observation will comprise the main part of the proof of \autoref{cor:rankTwo}, which we restate below for the reader's convenience.

\rankTwo*

We will prove \autoref{cor:rankTwo} for closed hyperbolic surface groups in $ \SL_{ 3 }^{ \pm } \of{ \R } $; the argument is identical for closed hyperbolic $ 3 $-manifold groups in $ \SL_{ 3 } \of{ \C } $.

\begin{proof}

Fix a reducible, discrete, and faithful representation $ \rho \colon \Gamma \to \SL_{ 3 }^{ \pm } \of{ \R } $ of a closed hyperbolic surface group $ \Gamma $. If $ \rho \of{ \Gamma } $ preserves a line in $ \R^{ 3 } $, note that $ \rho $ is conjugate to a discrete and faithful reducible suspension $ \rho_{ 1 , 2 } \of{ \varphi , \xi , \zeta , \kappa } $, for some homomorphism $ \varphi \colon \Gamma \to \R $, representations $ \xi \colon \Gamma \to \SL_{ 1 }^{ \pm } \of{ \R } $ and $ \zeta \colon \Gamma \to \SL_{ 2 }^{ \pm } \of{ \R } $, and a compatible $ \kappa \colon \Gamma \to \R^{ 1 \times 2 } $. Since $ \Gamma $ is centerless and $ \xi \of{ \Gamma } \subseteq \set{ \pm 1 } $ is abelian and therefore nilpotent, $ \zeta $ is discrete and faithful by \autoref{cor:discreteness}. As a discrete and faithful representation of a closed hyperbolic surface group in $ \SL_{ 2 }^{ \pm } \of{ \R } $, $ \zeta $ is a quasi-isometric embedding. Thus \autoref{prop:displacement} implies that $ \rho $ is a quasi-isometric embedding.

On the other hand, if $ \rho \of{ \Gamma } $ preserves a plane in $ \R^{ 3 } $, then the contragredient $ \contra{ \rho } \colon \Gamma \to \SL_{ 3 }^{ \pm } \of{ \R } $ is discrete and faithful, and it preserves a line in $ \R^{ 3 } $. The above argument implies that $ \contra{ \rho } $ is displacing, and hence so is $ \rho $, since $ \rho $ and $ \contra{ \rho } $ have the same translation lengths. \qedhere

\end{proof}

\section{Convexity} \label{sec:convexity}

In this section, we will prove \autoref{thm:convexity}, that the Anosov reducible suspensions of a fixed representation $ \Gamma \to \SL_{ q }^{ * } \of{ \K } $ correspond to a bounded convex domain in the finite-dimensional vector space $ \hom \of{ \Gamma , \R } $. To that end, fix for this section a homomorphism $ \varphi \colon \Gamma \to \R $, representations $ \xi \colon \Gamma \to U_{ p }^{ * } \of{ \K } $ and $ \zeta \colon \Gamma \to \SL_{ q }^{ * } \of{ \K } $, and a compatible $ \kappa \colon \Gamma \to \K^{ p \times q } $, and consider the reducible suspension $ \rho = \rho_{ p , q } \of{ \varphi , \xi , \zeta , \kappa } $.

Since we aim to use Kassel--Potrie's \cite[Corollary~4.6]{KP22} characterization of Anosov representations to determine when $ \rho $ is Anosov, it will prove vitally useful to determine for each $ \gamma \in \Gamma $ the eigenvalues of $ \rho \of{ \gamma } $, $ p $ of which come from the block $ e^{ \frac{ 1 }{ p } } \xi \of{ \gamma } $ and the other $ q $ of which come from the block $ e^{ -\frac{ 1 }{ q } \varphi \of{ \gamma } } \zeta \of{ \gamma } $. Specifically, $ \rho \of{ \gamma } $ has an eigenvalue of magnitude
\[
e^{ \frac{ 1 }{ p } \varphi \of{ \gamma } } \lambda_{ j } \of{ \xi \of{ \gamma } } = e^{ \frac{ 1 }{ p } \varphi \of{ \gamma } }
\]
for each integer $ 1 \leq j \leq p $ and an eigenvalue of magnitude $ e^{ -\frac{ 1 }{ q } \varphi \of{ \gamma } } \lambda_{ k } \of{ \zeta \of{ \gamma } } $ for each integer $ 1 \leq k \leq q $. We first observe that this implies that whether $ \rho $ is Anosov is independent of $ \xi $ and $ \kappa $.

\begin{lemma+} \label{lem:generality}

$ \rho $ is $ P_{ k } $-Anosov if and only if $ \rho_{ p , q } \of{ \varphi , \vect{ I }_{ p } , \zeta , \vect{ 0 } } $ is $ P_{ k } $-Anosov.

\begin{proof}

For any $ \gamma \in \Gamma $, the matrices $ \rho \of{ \gamma } $ and $ \rho_{ p , q } \of{ \varphi , \vect{ I }_{ p } , \zeta , \vect{ 0 } } \of{ \gamma } $ have the same eigenvalue magnitudes, counted with multiplicity; namely, each matrix has $ p $ eigenvalues of magnitude $ e^{ \frac{ 1 }{ p } } $ and also an eigenvalue of magnitude $ e^{ -\frac{ 1 }{ q } \varphi \of{ \gamma } } \lambda_{ j } \of{ \zeta \of{ \gamma } } $ for each integer $ 1 \leq j \leq q $. Thus $ \rho $ is $ P_{ k } $-Anosov if and only if $ \rho_{ p , q } \of{ \varphi , \vect{ I }_{ p } , \zeta , \vect{ 0 } } $ is. \qedhere

\end{proof}

\end{lemma+}

Note that \autoref{lem:generality} implies that $ A^{ p }_{ k } \of{ \zeta } $ (defined in \autoref{sec:introduction}) is well-defined for any representation $ \zeta \colon \Gamma \to \SL_{ q }^{ * } \of{ \K } $. We will use the following necessary bounds on the configuration of the eigenvalues of reducible suspensions of $ \zeta $ to prove \autoref{thm:convexity} and \autoref{cor:bounds}.

\begin{proposition+} \label{prop:eigenvalues}

If $ \rho $ is $ P_{ k } $-Anosov for some integer $ 1 \leq k \leq \frac{ p + q }{ 2 } $, then $ 1 \leq k \leq \frac{ q }{ 2 } $, and $ \zeta $ is $ P_{ k } $-Anosov. Moreover, for all $ \gamma \in \Gamma $,
\[
\lambda_{ q - k + 1 } \of{ \zeta \of{ \gamma } } \leq e^{ \frac{ p + q }{ p q } \varphi \of{ \gamma } } \leq \lambda_{ k } \of{ \zeta \of{ \gamma } } ,
\]
so that $ \lambda_{ k } \of{ \rho \of{ \gamma } } = e^{ -\frac{ 1 }{ q } \varphi \of{ \gamma } } \lambda_{ k } \of{ \zeta \of{ \gamma } } $. The above inequalities are strict for all infinite order $ \gamma \in \Gamma $.

\end{proposition+}

We note that the inequality in the above result serves to fix the relative configuration of the eigenvalue magnitudes of $ \rho \of{ \gamma } $, since it is equivalent to
\[
e^{ -\frac{ 1 }{ q } \varphi \of{ \gamma } } \lambda_{ q - k + 1 } \of{ \zeta \of{ \gamma } } \leq e^{ \frac{ 1 }{ p } \varphi \of{ \gamma } } \leq e^{ -\frac{ 1 }{ q } \varphi \of{ \gamma } } \lambda_{ k } \of{ \zeta \of{ \gamma } } .
\]
Specifically, the above inequality is strict if and only if $ e^{ \frac{ 1 }{ p } \varphi \of{ \gamma } } $ is among neither the $ k $ largest nor the $ k $ smallest eigenvalues of $ \rho \of{ \gamma } $. Otherwise, we will see that $ \rho $ cannot be $ P_{ k } $-Anosov.

\begin{proof} 

By \autoref{lem:generality}, we may assume without loss of generality that $ \xi \of{ \gamma } = \vect{ I }_{ p } $ and $ \kappa \of{ \gamma } = \vect{ 0 } $ for all $ \gamma \in \Gamma $, so that $ \rho $ has the formula
\[
\rho \of{ \gamma } = e^{ \frac{ 1 }{ p } \varphi \of{ \gamma } } \vect{ I }_{ p } \oplus e^{ -\frac{ 1 }{ q } \varphi \of{ \gamma } } \zeta \of{ \gamma } = \begin{pmatrix} e^{ \frac{ 1 }{ p } \varphi \of{ \gamma } } \vect{ I }_{ p } & \vect{ 0 } \\ \vect{ 0 } & e^{ -\frac{ 1 }{ q } \varphi \of{ \gamma } } \zeta \of{ \gamma } \end{pmatrix} .
\]
In particular, we note that for any $ \gamma \in \Gamma $, $ \rho \of{ \gamma } $ has $ p $ eigenvalues of magnitude $ e^{ \frac{ 1 }{ p } \varphi \of{ \gamma } } $ with corresponding eigenvectors $ \vect{ e }_{ 1 } , \dotsc , \vect{ e }_{ p } $, and it has one eigenvalue of magnitude $ e^{ -\frac{ 1 }{ q } \varphi \of{ \gamma } } \lambda_{ j } \of{ \zeta \of{ \gamma } } $ with generalized eigenvector in $ \lspan_{ \K } \of{ \vect{ e }_{ p + 1 } , \dotsc , \vect{ e }_{ p + q } } $ for each integer $ 1 \leq j \leq q $.

We claim that it suffices to show that for any infinite order $ \gamma \in \Gamma $, $ e^{ \frac{ 1 }{ p } \varphi \of{ \gamma } } $ is among neither the $ k $ largest nor the $ k $ smallest eigenvalues of $ \rho \of{ \gamma } $. Indeed, if this is true, then for all such infinite order $ \gamma $, the $ k $ largest and $ k $ smallest eigenvalues of $ \rho \of{ \gamma } $ are the (properly scaled) eigenvalues of $ \zeta \of{ \gamma } $, of which there are $ q $. More concretely, $ 2 k \leq q $, and $ e^{ -\frac{ 1 }{ q } \varphi \of{ \gamma } } \lambda_{ k } \of{ \zeta \of{ \gamma } } $ and $ e^{ -\frac{ 1 }{ q } \varphi \of{ \gamma } } \lambda_{ q - k + 1 } \of{ \zeta \of{ \gamma } } $ are the $ k $th largest and smallest eigenvalues of $ \rho \of{ \gamma } $. Thus
\[
e^{ -\frac{ 1 }{ q } \varphi \of{ \gamma } } \lambda_{ q - k + 1 } \of{ \zeta \of{ \gamma } } = \lambda_{ p + q - k + 1 } \of{ \rho \of{ \gamma } } < e^{ \frac{ 1 }{ p } \varphi \of{ \gamma } } < \lambda_{ k } \of{ \rho \of{ \gamma } } = e^{ -\frac{ 1 }{ q } \varphi \of{ \gamma } } \lambda_{ k } \of{ \zeta \of{ \gamma } } .
\]
Multiplying the above inequalities by $ e^{ \frac{ 1 }{ q } \varphi \of{ \gamma } } $ gives the desired result for infinite order elements.

On the other hand, if $ \gamma \in \Gamma $ has finite order, then its eigenvalues all have magnitude $ 1 $. Moreover, $ \varphi \of{ \gamma } = 0 $, since $ \R $ is torsion-free, and so
\[
\lambda_{ q - k } \of{ \rho \of{ \gamma } } = e^{ \frac{ p + q }{ p q } \varphi \of{ \gamma } } = \lambda_{ k } \of{ \rho \of{ \gamma } } = 1 .
\]

Thus it suffices to show that for any infinite order $ \gamma \in \Gamma $, $ e^{ \frac{ 1 }{ p } \varphi \of{ \gamma } } $ is among neither the $ k $ largest nor the $ k $ smallest eigenvalues of $ \rho \of{ \gamma } $. To that end, let $ \xi_{ \rho }^{ k } \colon \partial \Gamma \to \Gr_{ k } \of{ \K^{ p + q } } $ denote the $ P_{ k } $-Anosov limit map for $ \rho $, and consider the following subsets
\begin{align*}
A & = \set{ V \in \Gr_{ k } \of{ \K^{ p + q } } : \lspan_{ \K } \of{ \vect{ e }_{ 1 } , \dotsc , \vect{ e }_{ p } } \subseteq V } \\
B & = \set{ V \in \Gr_{ k } \of{ \K^{ p + q } } : V \subseteq \lspan_{ \K } \of{ \vect{ e }_{ p + 1 } , \dotsc , \vect{ e }_{ p + q } } }
\end{align*}
of the Grassmannian $ \Gr_{ k } \of{ \K^{ p + q } } $. We note that $ A $ and $ B $ are closed and disjoint.

Since $ \rho $ is $ P_{ k } $-Anosov, $ \rho \of{ \gamma } $ is $ P_{ k } $-biproximal, meaning that
\begin{align*}
\lambda_{ k } \of{ \rho \of{ \gamma } } & > \lambda_{ k + 1 } \of{ \rho \of{ \gamma } } & \lambda_{ p + q - k + 1 } \of{ \rho \of{ \gamma } } = \frac{ 1 }{ \lambda_{ k } \of{ \rho \of{ \gamma } } } < \frac{ 1 }{ \lambda_{ k + 1 } \of{ \rho \of{ \gamma } } } = \lambda_{ p + q - k } \of{ \rho \of{ \gamma } } .
\end{align*}
Moreover, $ \xi_{ \rho }^{ k } \of{ \gamma^{ + } } = \rho \of{ \gamma }^{ + }_{ k } $ and $ \xi_{ \rho }^{ k } \of{ \gamma^{ - } } = \rho \of{ \gamma }^{ - }_{ k } $ are spanned by the generalized eigenvectors corresponding to the $ k $ largest and $ k $ smallest eigenvalues of $ \rho \of{ \gamma } $. Since $ \vect{ e }_{ 1 } , \dotsc , \vect{ e }_{ p } $ have eigenvalues of the same magnitude, both of the aforementioned disjoint lists of generalized eigenvectors contain either all $ \vect{ e }_{ 1 } , \dotsc , \vect{ e }_{ p } $ or none of them. We conclude that for such an infinite order $ \gamma \in \Gamma $, precisely one of the following three cases occur:
\begin{enumerate}

\item[(i)]

Suppose that $ e^{ \frac{ 1 }{ p } \varphi \of{ \gamma } } $ is among the $ k $ largest eigenvalues of $ \rho \of{ \gamma } $. Then $ \lspan_{ \K } \of{ \vect{ e }_{ 1 } , \dotsc , \vect{ e }_{ p } } \subseteq \xi_{ \rho }^{ k } \of{ \gamma^{ + } } $, and so $ \xi_{ \rho }^{ k } \of{ \gamma^{ + } } \in A $. In fact, the same argument applies to any conjugates of $ \gamma $, since $ e^{ \frac{ 1 }{ p } \varphi \of{ \delta \gamma \delta^{ -1 } } } = e^{ \frac{ 1 }{ p } \varphi \of{ \gamma } } $ is among the $ k $ largest eigenvalues of $ \rho \of{ \delta \gamma \delta^{ -1 } } $. Thus
\[
\xi_{ \rho }^{ k } \of{ \delta \cdot \gamma^{ + } } = \xi_{ \rho }^{ k } \of{ \of{ \delta \gamma \delta^{ -1 } }^{ + } } = \rho \of{ \delta \gamma \delta^{ -1 } }^{ + }_{ k } \supseteq \lspan_{ \K } \of{ \vect{ e }_{ 1 } , \dotsc , \vect{ e }_{ p } }
\]
for all $ \delta \in \Gamma $, and so $ \xi_{ \rho }^{ k } \of{ \Gamma \cdot \gamma^{ + } } \subseteq A $.

On the other hand, $ e^{ \frac{ 1 }{ p } \varphi \of{ \gamma } } $ is not among the $ k $ smallest eigenvalues of $ \rho \of{ \gamma } $, so the generalized eigenvectors corresponding to said eigenvectors are in the complementary subspace; that is, $ \xi_{ \rho }^{ k } \of{ \gamma^{ - } } \subseteq \lspan_{ \K } \of{ \vect{ e }_{ p + 1 } , \dotsc , \vect{ e }_{ p + q } } $, and so $ \xi_{ \rho }^{ k } \of{ \gamma^{ - } } \in B $.

\item[(ii)]

Now suppose that $ e^{ \frac{ 1 }{ p } \varphi \of{ \gamma } } $ is among the $ k $ smallest eigenvalues of $ \rho \of{ \gamma } $. Then $ e^{ \frac{ 1 }{ p } \varphi \of{ \gamma^{ -1 } } } $ is among the $ k $ largest eigenvalues of $ \rho \of{ \gamma^{ -1 } } $, and so the above argument implies that $ \xi_{ \rho }^{ k } \of{ \Gamma \cdot \gamma^{ - } } \subseteq A $ and $ \xi_{ \rho }^{ k } \of{ \gamma^{ + } } \in B $.

\item[(iii)]

Finally, suppose that $ e^{ \frac{ 1 }{ p } \varphi \of{ \gamma } } $ is among neither the $ k $ largest eigenvalues nor the $ k $ smallest eigenvalues of $ \rho \of{ \gamma } $. Then the generalized eigenvectors corresponding to said eigenvectors are in the complementary subspace; that is, $ \xi_{ \rho }^{ k } \of{ \gamma^{ \pm } } \subseteq \lspan_{ \K } \of{ \vect{ e }_{ p + 1 } , \dotsc , \vect{ e }_{ p + q } } $, and so $ \xi_{ \rho }^{ k } \of{ \gamma^{ + } } , \xi_{ \rho }^{ k } \of{ \gamma^{ - } } \in B $. 

\end{enumerate}

Note that in the first two cases, there is some $ \gamma_{ * } \in \set{ \gamma , \gamma^{ -1 } } $ so that $ \xi_{ \rho }^{ k } \of{ \Gamma \cdot \gamma_{ * }^{ + } } \subseteq A $ but $ \xi_{ \rho }^{ k } \of{ \gamma_{ * }^{ - } } \in B $. Since the orbit of $ \gamma_{ * }^{ + } $ is dense in $ \partial \Gamma $ and $ A $ is closed in $ \Gr_{ k } \of{ \K^{ p + q } } $,
\begin{align*}
\xi_{ \rho }^{ k } \of{ \gamma_{ * }^{ - } } \in \xi_{ \rho }^{ k } \of{ \partial \Gamma } = \xi_{ \rho }^{ k } \of{ \overline{ \Gamma \cdot \gamma_{ * }^{ + } } } \subseteq \overline{ \xi_{ \rho }^{ k } \of{ \Gamma \cdot \gamma_{ * }^{ + } } } \subseteq \overline{ A } = A .
\end{align*}
Since $ A $ and $ B $ are disjoint, this contradiction implies that for all infinite order $ \gamma \in \Gamma $, $ e^{ \frac{ 1 }{ p } \varphi \of{ \gamma } } $ is among neither the $ k $ largest eigenvalues nor the $ k $ smallest eigenvalues of $ \rho \of{ \gamma } $. \qedhere

\end{proof}

We are now ready to prove \autoref{thm:convexity}, which we restate below for the reader's convenience.

\convexity*

\begin{proof}

If $ \zeta $ is not $ P_{ k } $-Anosov, then the contrapositive of \autoref{prop:eigenvalues} implies that $ A^{ p }_{ k } \of{ \zeta } $ is empty and therefore vacuously bounded, convex, and open. Thus we may assume that $ \zeta $ is $ P_{ k } $-Anosov.

\autoref{thm:stability} implies that any Anosov reducible suspension of $ \zeta $ has a neighborhood in $ \hom \of{ \Gamma , \SL_{ d }^{ * } \of{ \K } } $ consisting of Anosov representations. Since the map $ \varphi \mapsto \rho_{ p , q } \of{ \varphi , \vect{ I }_{ p } , \zeta , \vect{ 0 } } $ is continuous, this implies that $ A^{ p }_{ k } \of{ \zeta } $ is open in $ \hom \of{ \Gamma , \R } $.

For convexity, fix homomorphisms $ \varphi_{ 0 } , \varphi_{ 1 } \in A^{ p }_{ k } \of{ \zeta } $, so that the reducible suspensions $ \rho_{ 0 } = \rho_{ p , q } \of{ \varphi_{ 0 } , \vect{ I }_{ p } , \zeta , \vect{ 0 } } $ and $ \rho_{ 1 } = \rho_{ p , q } \of{ \varphi_{ 1 } , \vect{ I }_{ p } , \zeta , \vect{ 0 } } $ are $ P_{ k } $-Anosov. Then there are constants $ a_{ 0 } , a_{ 1 } > 0 $ and $ b_{ 0 } , b_{ 1 } \geq 0 $ so that
\begin{align*}
\log \of{ \frac{ \lambda_{ k } \of{ \rho_{ 0 } \of{ \gamma } } }{ \lambda_{ k + 1 } \of{ \rho_{ 0 } \of{ \gamma } } } } & \geq a_{ 0 } \norm{ \gamma }_{ \Sigma } - b_{ 0 } & \log \of{ \frac{ \lambda_{ k } \of{ \rho_{ 1 } \of{ \gamma } } }{ \lambda_{ k + 1 } \of{ \rho_{ 1 } \of{ \gamma } } } } & \geq a_{ 1 } \norm{ \gamma }_{ \Sigma } - b_{ 1 }
\end{align*}
for all $ \gamma \in \Gamma $. Let
\begin{align*}
\varphi_{ t } & = \of{ 1 - t } \varphi_{ 0 } + t \varphi_{ 1 } & a_{ t } & = \of{ 1 - t } a_{ 0 } + t a_{ 1 } > 0 \\ \\
\rho_{ t } & = \rho_{ p , q } \of{ \varphi_{ t } , \vect{ I }_{ p } , \zeta , \vect{ 0 } } & b_{ t } & = \of{ 1 - t } b_{ 0 } + t b_{ 1 } \geq 0
\end{align*}
for each $ 0 < t < 1 $, and note that \autoref{prop:eigenvalues} implies that
\begin{align*}
\lambda_{ q - k } \of{ \zeta \of{ \gamma } } & \leq e^{ \frac{ p + q }{ p q } \varphi_{ 0 } \of{ \gamma } } \leq \lambda_{ k } \of{ \zeta \of{ \gamma } } & \lambda_{ q - k } \of{ \zeta \of{ \gamma } } & \leq e^{ \frac{ p + q }{ p q } \varphi_{ 1 } \of{ \gamma } } \leq \lambda_{ k } \of{ \zeta \of{ \gamma } } 
\end{align*}
for all $ \gamma \in \Gamma $. This implies that the line segment
\[
\set{ e^{ \frac{ p + q }{ p q } \varphi_{ t } \of{ \gamma } } \in \R : 0 \leq t \leq 1 } = \clint{ \min \of{ e^{ \frac{ p + q }{ p q } \varphi_{ 0 } \of{ \gamma } } , e^{ \frac{ p + q }{ p q } \varphi_{ 1 } \of{ \gamma } } } }{ \max \of{ e^{ \frac{ p + q }{ p q } \varphi_{ 0 } \of{ \gamma } } , e^{ \frac{ p + q }{ p q } \varphi_{ 1 } \of{ \gamma } } } }
\]
is contained in the closed interval between $ \lambda_{ k } \of{ \zeta \of{ \gamma } } $ and $ \lambda_{ q - k } \of{ \zeta \of{ \gamma } } $, and so
\[
\lambda_{ j } \of{ \rho_{ t } \of{ \gamma } } = e^{ -\frac{ 1 }{ q } \varphi_{ t } \of{ \gamma } } \lambda_{ j } \of{ \zeta \of{ \gamma } }
\]
for all integers $ 1 \leq j \leq k $, $ 0 \leq t \leq 1 $, and $ \gamma \in \Gamma $. Note that
\begin{align*}
\log \of{ \frac{ e^{ -\frac{ 1 }{ q } \varphi_{ t } \of{ \gamma } } \lambda_{ k } \of{ \zeta \of{ \gamma } } }{ e^{ \frac{ 1 }{ p } \varphi_{ t } \of{ \gamma } } } } & = \log \of{ \frac{ \lambda_{ k } \of{ \zeta \of{ \gamma } } }{ e^{ \frac{ p + q }{ p q } \varphi_{ t } \of{ \gamma } } } } = \log \of{ \lambda_{ k } \of{ \zeta \of{ \gamma } } } - \frac{ p + q }{ p q } \varphi_{ t } \of{ \gamma } \\
& = \of{ 1 - t } \of{ \log \of{ \lambda_{ k } \of{ \zeta \of{ \gamma } } } - \frac{ p + q }{ p q } \varphi_{ 0 } \of{ \gamma } } + t \of{ \log \of{ \lambda_{ k } \of{ \zeta \of{ \gamma } } } - \frac{ p + q }{ p q } \varphi_{ 1 } \of{ \gamma } } \\
& = \of{ 1 - t } \log \of{ \frac{ \lambda_{ k } \of{ \zeta \of{ \gamma } } }{ e^{ \frac{ p + q }{ p q } \varphi_{ 0 } \of{ \gamma } } } } + t \log \of{ \frac{ \lambda_{ k } \of{ \zeta \of{ \gamma } } }{ e^{ \frac{ p + q }{ p q } \varphi_{ 1 } \of{ \gamma } } } } \\
& \geq \of{ 1 - t } \log \of{ \frac{ \lambda_{ k } \of{ \rho_{ 0 } \of{ \gamma } } }{ \lambda_{ k + 1 } \of{ \rho_{ 0 } \of{ \gamma } } } } + t \log \of{ \frac{ \lambda_{ k } \of{ \rho_{ 1 } \of{ \gamma } } }{ \lambda_{ k + 1 } \of{ \rho_{ 1 } \of{ \gamma } } } } \\
& \geq \of{ 1 - t } \of{ a_{ 0 } \norm{ \gamma }_{ \Sigma } - b_{ 0 } } + t \of{ a_{ 1 } \norm{ \gamma }_{ \Sigma } - b_{ 1 } } = a_{ t } \norm{ \gamma }_{ \Sigma } - b_{ t }
\end{align*}
grows at least linearly in $ \norm{ \gamma }_{ \Sigma } $. On the other hand,
\[
\log \of{ \frac{ e^{ -\frac{ 1 }{ q } \varphi_{ t } \of{ \gamma } } \lambda_{ k } \of{ \zeta \of{ \gamma } } }{ e^{ -\frac{ 1 }{ q } \varphi_{ t } \of{ \gamma } } \lambda_{ k + 1 } \of{ \zeta \of{ \gamma } } } } = \log \of{ \frac{ \lambda_{ k } \of{ \zeta \of{ \gamma } } }{ \lambda_{ k + 1 } \of{ \zeta \of{ \gamma } } } }
\]
also grows at least linearly in $ \norm{ \gamma }_{ \Sigma } $, since $ \zeta \colon \Gamma \to \SL_{ q }^{ * } \of{ \K } $ is $ P_{ k } $-Anosov. Therefore,
\[
\log \of{ \frac{ \lambda_{ k } \of{ \rho_{ t } \of{ \gamma } } }{ \lambda_{ k + 1 } \of{ \rho_{ t } \of{ \gamma } } } } = \min \of{ \log \of{ \frac{ e^{ -\frac{ 1 }{ q } \varphi_{ t } \of{ \gamma } } \lambda_{ k } \of{ \zeta \of{ \gamma } } }{ e^{ \frac{ 1 }{ p } \varphi \of{ \gamma } } } } , \log \of{ \frac{ e^{ -\frac{ 1 }{ q } \varphi_{ t } \of{ \gamma } } \lambda_{ k } \of{ \zeta \of{ \gamma } } }{ e^{ -\frac{ 1 }{ q } \varphi_{ t } \of{ \gamma } } \lambda_{ k + 1 } \of{ \zeta \of{ \gamma } } } } }
\]
grows at least linearly in $ \norm{ \gamma }_{ \Sigma } $, and so $ \rho_{ t } $ is $ P_{ k } $-Anosov for all $ 0 \leq t \leq 1 $. Thus $ \varphi_{ t } \in A^{ p }_{ k } \of{ \zeta } $ for all $ 0 \leq t \leq 1 $, and so $ A^{ p }_{ k } \of{ \zeta } $ is convex.

Finally, suppose for a contradiction that $ A^{ p }_{ k } \of{ \zeta } $ is unbounded. As an unbounded convex domain in a finite-dimensional real vector space, it contains a ray $ \set{ \varphi + t \psi \in \hom \of{ \Gamma , \R } : t \geq 0 } $. Choose $ \gamma \in \Gamma $ so that $ \psi \of{ \gamma } > 0 $, and note that $ e^{ \frac{ p + q }{ p q } \of{ \varphi \of{ \gamma } + t \psi \of{ \gamma } } } > \lambda_{ k } \of{ \zeta \of{ \gamma } } $ whenever
\[
t > \frac{ 1 }{ \psi \of{ \gamma } } \of{ \frac{ p q }{ p + q } \log \of{ \lambda_{ k } \of{ \zeta \of{ \gamma } } } - \varphi \of{ \gamma } } .
\]
\autoref{prop:eigenvalues} implies that $ \rho_{ p , q } \of{ \varphi + t \psi , \vect{ I }_{ p } , \zeta , \vect{ 0 } } $ is not $ P_{ k } $-Anosov for all such $ t $, so that $ \varphi + t \psi \notin A^{ p }_{ k } \of{ \zeta } $. This contradiction implies that $ A^{ p }_{ k } \of{ \zeta } $ is bounded. \qedhere

\qedhere

\end{proof}

\section{Reducible Suspensions of Symmetric Anosov Representations} \label{sec:symmetricAnosov}

Recall that for an integer $ 1 \leq k \leq \frac{ q }{ 2 } $, we call a representation $ \zeta \colon \Gamma \to \SL_{ q }^{ * } \of{ \K } $ \emph{symmetric $ P_{ k } $-Anosov} if it is $ P_{ k } $-Anosov and $ \lambda_{ k } \of{ \zeta \of{ \gamma } } = \lambda_{ k } \of{ \zeta \of{ \gamma^{ -1 } } } $ for all $ \gamma \in \Gamma $. In this section, we will use this property to build on the results of the previous section in precisely characterizing when reducible suspensions of symmetric Anosov representations are themselves Anosov. To that end, fix for this section a homomorphism $ \varphi \colon \Gamma \to \R $, representations $ \xi \colon \Gamma \to U_{ p }^{ * } \of{ \K } $ and $ \zeta \colon \Gamma \to \SL_{ q }^{ * } \of{ \K } $, and a compatible map $ \kappa \colon \Gamma \to \K^{ p \times q } $, and consider the reducible suspension $ \rho = \rho_{ p , q } \of{ \varphi , \xi , \zeta , \kappa } $. We begin by proving a slightly stronger version of \autoref{prop:eigenvalues} for reducible suspensions of symmetric representations.
\begin{proposition+} \label{prop:symmetricAnosov}

Suppose that $ \zeta $ is symmetric $ P_{ k } $-Anosov for some integer $ 1 \leq k \leq \frac{ q }{ 2 } $. If $ \rho $ is $ P_{ k } $-Anosov, then
\[
\lambda_{ q - k + 1 } \of{ \zeta \of{ \gamma } } \leq e^{ -\frac{ p + q }{ p q } \abs{ \varphi \of{ \gamma } } } \leq e^{ \frac{ p + q }{ p q } \abs{ \varphi \of{ \gamma } } } \leq \lambda_{ k } \of{ \zeta \of{ \gamma } }
\]
for all $ \gamma \in \Gamma $. Moreover, the first and last above inequalities are strict for all infinite order $ \gamma \in \Gamma $.

\begin{proof}

Since $ \rho $ is $ P_{ k } $-Anosov, \autoref{prop:eigenvalues} implies that
\[
\lambda_{ q - k + 1 } \of{ \zeta \of{ \gamma } } \leq e^{ \frac{ p + q }{ p q } \varphi \of{ \gamma } } \leq \lambda_{ k } \of{ \zeta \of{ \gamma } }
\]
for all $ \gamma \in \Gamma $, and that the above inequalities are strict for infinite order $ \gamma \in \Gamma $. Applying this to $ \gamma^{ -1 } $, we see that
\[
\lambda_{ q - k + 1 } \of{ \zeta \of{ \gamma } } = \frac{ 1 }{ \lambda_{ k } \of{ \zeta \of{ \gamma^{ -1 } } } } = \frac{ 1 }{ \lambda_{ k } \of{ \zeta \of{ \gamma } } } = \lambda_{ q - k + 1 } \of{ \zeta \of{ \gamma^{ -1 } } } \leq e^{ -\frac{ p + q }{ p q } \varphi \of{ \gamma } } \leq \lambda_{ k } \of{ \zeta \of{ \gamma^{ -1 } } } = \lambda_{ k } \of{ \zeta \of{ \gamma } }
\]
for all $ \gamma \in \Gamma $, and that the above inequalities are strict for infinite order $ \gamma \in \Gamma $. \qedhere

\end{proof}

\end{proposition+}

We are now ready to prove \autoref{thm:symmetricAnosov}, which we restate below for the reader's convenience.

\symmetricAnosov*

\begin{proof}

First suppose that
\[
\inf_{ \varphi \of{ \gamma } \neq 0 }{ \frac{ \log \of{ \lambda_{ k } \of{ \zeta \of{ \gamma } } } }{ \abs{ \varphi \of{ \gamma } } } } > \frac{ p + q }{ p q } .
\]
Then there is some constant $ a > \frac{ p + q }{ p q } $ so that $ \log \of{ \lambda_{ k } \of{ \zeta \of{ \gamma } } } \geq a \abs{ \varphi \of{ \gamma } } $ whenever $ \varphi \of{ \gamma } \neq 0 $. In fact, since $ \rho $ is symmetric $ P_{ k } $-Anosov, this inequality also holds when $ \varphi \of{ \gamma } = 0 $, so that
\[
\log \of{ \frac{ e^{ -\frac{ 1 }{ q } \varphi \of{ \gamma } } \lambda_{ k } \of{ \zeta \of{ \gamma } } }{ e^{ \frac{ 1 }{ p } \varphi \of{ \gamma } } } } = \log \of{ \lambda_{ k } \of{ \zeta \of{ \gamma } } } - \frac{ p + q }{ p q } \varphi \of{ \gamma } \geq \of{ 1 - \frac{ p + q }{ a p q } } \log \of{ \lambda_{ k } \of{ \zeta \of{ \gamma } } }
\]
for all $ \gamma \in \Gamma $. Our assumption that $ a > \frac{ p + q }{ p q } $ implies that the above coefficient of $ \log \of{ \lambda_{ k } \of{ \zeta \of{ \gamma } } } $ is positive. Moreover, for any infinite order $ \gamma \in \Gamma $,
\[
e^{ -\frac{ 1 }{ q } \varphi \of{ \gamma } } \lambda_{ k } \of{ \zeta \of{ \gamma } } \geq e^{ \of{ a - \frac{ 1 }{ q } } \abs{ \varphi \of{ \gamma } } } > e^{ \of{ \frac{ p + q }{ p q } - \frac{ 1 }{ q } } \abs{ \varphi \of{ \gamma } } } = e^{ \frac{ 1 }{ p } \abs{ \varphi \of{ \gamma } } } \geq e^{ \frac{ 1 }{ p } \varphi \of{ \gamma } } ,
\]
and thus $ \lambda_{ j } \of{ \rho \of{ \gamma } } = e^{ -\frac{ 1 }{ q } \varphi \of{ \gamma } } \lambda_{ j } \of{ \zeta \of{ \gamma } } $ for all integers $ 1 \leq j \leq k $. We see that
\begin{align*}
\log \of{ \frac{ \lambda_{ k } \of{ \rho \of{ \gamma } } }{ \lambda_{ k + 1 } \of{ \rho \of{ \gamma } } } } & = \min \of{ \log \of{ \frac{ e^{ -\frac{ 1 }{ q } \varphi \of{ \gamma } } \lambda_{ k } \of{ \zeta \of{ \gamma } } }{ e^{ \frac{ 1 }{ p } \varphi \of{ \gamma } } } } , \log \of{ \frac{ e^{ -\frac{ 1 }{ q } \varphi \of{ \gamma } } \lambda_{ k } \of{ \zeta \of{ \gamma } } }{ e^{ -\frac{ 1 }{ q } \varphi \of{ \gamma } } \lambda_{ k + 1 } \of{ \zeta \of{ \gamma } } } } } \\
& \geq \min \of{ \of{ 1 - \frac{ p + q }{ a p q } } \log \of{ \lambda_{ k } \of{ \zeta \of{ \gamma } } } , \log \of{ \frac{ \lambda_{ k } \of{ \zeta \of{ \gamma } } }{ \lambda_{ k + 1 } \of{ \zeta \of{ \gamma } } } } } .
\end{align*}
Since $ \zeta $ is symmetric $ P_{ k } $-Anosov, \autoref{lem:symmetricAnosov} implies that the above grows at least linearly in $ \norm{ \gamma }_{ \Sigma } $, and so $ \rho $ is $ P_{ k } $-Anosov; that is, $ \varphi \in A^{ p }_{ k } \of{ \zeta } $.

Now suppose that
\[
\inf_{ \varphi \of{ \gamma } \neq 0 }{ \frac{ \log \of{ \lambda_{ k } \of{ \zeta \of{ \gamma } } } }{ \abs{ \varphi \of{ \gamma } } } } < \frac{ p + q }{ p q } .
\]
Then $ \log \of{ \lambda_{ k } \of{ \zeta \of{ \gamma } } } < \frac{ p + q }{ p q } \abs{ \varphi \of{ \gamma } } $ for some $ \gamma \notin \ker \of{ \varphi } $, so that $ \lambda_{ k } \of{ \zeta \of{ \gamma } } < e^{ \frac{ p + q }{ p q } \abs{ \varphi \of{ \gamma } } } $. \autoref{prop:symmetricAnosov} now implies that $ \rho $ is not $ P_{ k } $-Anosov; that is, $ \varphi \notin A^{ p }_{ k } \of{ \zeta } $.

Finally, note that if
\[
\inf_{ \varphi \of{ \gamma } \neq 0 }{ \frac{ \log \of{ \lambda_{ k } \of{ \zeta \of{ \gamma } } } }{ \abs{ \varphi \of{ \gamma } } } } = \frac{ p + q }{ p q } ,
\]
then $ \rho = \lim_{ n \to \infty }{ \rho_{ p , q } \of{ \frac{ n + 1 }{ n } \varphi , \vect{ I }_{ p } , \zeta , \vect{ 0 } } } $ is a limit of representations which are not $ P_{ k } $-Anosov by the above argument, since
\[
\inf_{ \frac{ n + 1 }{ n } \varphi \of{ \gamma } \neq 0 }{ \frac{ \log \of{ \lambda_{ k } \of{ \zeta \of{ \gamma } } } }{ \abs{ \frac{ n + 1 }{ n } \varphi \of{ \gamma } } } } = \frac{ n }{ n + 1 } \inf_{ \varphi \of{ \gamma } \neq 0 }{ \frac{ \log \of{ \lambda_{ k } \of{ \zeta \of{ \gamma } } } }{ \abs{ \varphi \of{ \gamma } } } } < \frac{ p + q }{ p q } .
\]
\autoref{thm:stability} now implies that $ \rho $ is itself not $ P_{ k } $-Anosov, so that $ \varphi \notin A^{ p }_{ k } \of{ \zeta } $. \qedhere

\end{proof}

\Fuchsian*

\begin{proof}

This follows immediately from \autoref{thm:symmetricAnosov} and the facts that $ \lambda_{ k } \of{ \iota_{ q } \of{ \eta \of{ \gamma } } } = \lambda_{ 1 } \of{ \eta \of{ \gamma } }^{ q - 2 k + 1 } $ and $ \frac{ p + q }{ p q \of{ q - 2 k + 1 } } $ is strictly decreasing in $ p $ and $ q $ and strictly increasing in $ k $. \qedhere

\end{proof}

The remainder of this section is dedicated to the proof of \autoref{cor:bounds}. Recall that we have fixed a generating set $ \Sigma $, which induces a uniform norm $ \norm{ \oparg }_{ \infty } $ on $ \hom \of{ \Gamma , \R } $, and that for a $ P_{ k } $-symmetric representation $ \zeta \colon \Gamma \to \SL_{ q }^{ * } \of{ \K } $, $ s_{ k } \of{ \zeta } $ denotes the optimal slope for the linear growth of $ \log \of{ \lambda_{ k } \of{ \zeta \of{ \gamma } } } $ in terms of $ \norm{ \gamma }_{ \Sigma } $; that is,
\[
s_{ k } \of{ \zeta } = \sup \set{ a \in \opint{ 0 }{ \infty } : \textrm{There is some } b \geq 0 \textrm{ so that } \log \of{ \lambda_{ k } \of{ \zeta \of{ \gamma } } } \geq a \norm{ \gamma }_{ \Sigma } - b \textrm{ for all } \gamma \in \Gamma } .
\]

\begin{lemma+} \label{lem:bounds}

If $ \zeta $ is symmetric $ P_{ k } $-Anosov and $ \varphi $ is not identically zero, then
\[
\frac{ s_{ k } \of{ \zeta } }{ \norm{ \varphi }_{ \infty } } \leq \inf_{ \varphi \of{ \gamma } \neq 0 }{ \frac{ \log \of{ \lambda_{ k } \of{ \zeta \of{ \gamma } } } }{ \abs{ \varphi \of{ \gamma } } } } \leq \frac{ \max_{ \sigma \in \Sigma }{ \log \of{ \lambda_{ k } \of{ \zeta \of{ \sigma } } } } }{ \norm{ \varphi }_{ \infty } } .
\]

\begin{proof}

For the upper bound, note that there is a generator $ \sigma_{ * } \in \Sigma $ so that $ \abs{ \varphi \of{ \sigma_{ * } } } = \norm{ \varphi }_{ \infty } > 0 $, which implies that
\[
\inf_{ \varphi \of{ \gamma } \neq 0 }{ \frac{ \log \of{ \lambda_{ k } \of{ \zeta \of{ \gamma } } } }{ \abs{ \varphi \of{ \gamma } } } } \leq \frac{ \log \of{ \lambda_{ k } \of{ \zeta \of{ \sigma_{ * } } } } }{ \abs{ \varphi \of{ \sigma_{ * } } } } = \frac{ \log \of{ \lambda_{ k } \of{ \zeta \of{ \sigma_{ * } } } } }{ \norm{ \varphi }_{ \infty } } \leq \frac{ \max_{ \sigma \in \Sigma }{ \log \of{ \lambda_{ k } \of{ \zeta \of{ \sigma } } } } }{ \norm{ \varphi }_{ \infty } } .
\]
For the lower bound, let $ 0 < a < s_{ k } \of{ \zeta } $, so that there is a constant $ b \geq 0 $ so that $ \log \of{ \lambda_{ k } \of{ \zeta \of{ \gamma } } } \geq a \norm{ \gamma }_{ \Sigma } - b $ for all $ \gamma \in \Gamma $. We claim that there is a sequence $ \of{ \gamma_{ n } }_{ n = 0 }^{ \infty } $ in $ \Gamma \setminus \ker \of{ \varphi } $ so that
\begin{align} 
\lim_{ n \to \infty }{ \frac{ \log \of{ \lambda_{ k } \of{ \zeta \of{ \gamma_{ n } } } } }{ \abs{ \varphi \of{ \gamma_{ n } } } } } & = \inf_{ \varphi \of{ \gamma } \neq 0 }{ \frac{ \log \of{ \lambda_{ k } \of{ \zeta \of{ \gamma } } } }{ \abs{ \varphi \of{ \gamma } } } } & \lim_{ n \to \infty }{ \norm{ \gamma_{ n } }_{ \Sigma } } & = \infty . \tag{$ \ast $}
\end{align}
Indeed, if such a sequence exists, then
\begin{align*}
\inf_{ \varphi \of{ \gamma } \neq 0 }{ \frac{ \log \of{ \lambda_{ k } \of{ \zeta \of{ \gamma } } } }{ \abs{ \varphi \of{ \gamma } } } } & = \lim_{ n \to \infty }{ \frac{ \log \of{ \lambda_{ k } \of{ \zeta \of{ \gamma_{ n } } } } }{ \abs{ \varphi \of{ \gamma_{ n } } } } } \geq \liminf_{ n \to \infty }{ \frac{ a \norm{ \gamma_{ n } }_{ \Sigma } - b }{ \norm{ \varphi }_{ \infty } \cdot \norm{ \gamma_{ n } }_{ \Sigma } } } \\
& \geq \liminf_{ x \to \infty }{ \frac{ a x - b }{ \norm{ \varphi }_{ \infty } x } } = \frac{ a }{ \norm{ \varphi }_{ \infty } }
\end{align*}
Since this holds for all $ 0 < a < s_{ k } \of{ \zeta } $,
\[
\inf_{ \varphi \of{ \gamma } \neq 0 }{ \frac{ \log \of{ \lambda_{ k } \of{ \zeta \of{ \gamma } } } }{ \abs{ \varphi \of{ \gamma } } } } \geq \frac{ s_{ k } \of{ \zeta } }{ \norm{ \varphi }_{ \infty } } .
\]
Thus it suffices to exhibit the existence of a sequence with the properties $ \of{ \ast } $. To that end, note that if
\[
\frac{ \log \of{ \lambda_{ k } \of{ \zeta \of{ \gamma_{ * } } } } }{ \abs{ \varphi \of{ \gamma_{ * } } } } = \inf_{ \varphi \of{ \gamma } \neq 0 }{ \frac{ \log \of{ \lambda_{ k } \of{ \zeta \of{ \gamma } } } }{ \abs{ \varphi \of{ \gamma } } } }
\]
for some $ \gamma_{ * } \notin \ker \of{ \varphi } $, then we may take $ \gamma_{ n } = \gamma_{ * }^{ n } $, since
\[
\lim_{ n \to \infty }{ \frac{ \log \of{ \lambda_{ k } \of{ \zeta \of{ \gamma_{ * }^{ n } } } } }{ \abs{ \varphi \of{ \gamma_{ * }^{ n } } } } } = \lim_{ n \to \infty }{ \frac{ \abs{ n } \cdot \log \of{ \lambda_{ k } \of{ \zeta \of{ \gamma_{ * } } } } }{ \abs{ n } \cdot \abs{ \varphi \of{ \gamma_{ * } } } } } = \frac{ \log \of{ \lambda_{ k } \of{ \zeta \of{ \gamma_{ * } } } } }{ \abs{ \varphi \of{ \gamma_{ * } } } } = \inf_{ \varphi \of{ \gamma } \neq 0 }{ \frac{ \log \of{ \lambda_{ k } \of{ \zeta \of{ \gamma } } } }{ \abs{ \varphi \of{ \gamma } } } } .
\]
Moreover, such a $ \gamma_{ * } $ has infinite order in $ \Gamma $, since $ \R $ is torsion-free and hence all finite order elements are in $ \ker \of{ \varphi } $. This implies that $ \lim_{ n \to \infty }{ \norm{ \gamma_{ * }^{ n } }_{ \Sigma } } = \infty $.

Otherwise, if no such $ \gamma_{ * } $ exists, then any sequence $ \of{ \gamma_{ n } }_{ n = 0 }^{ \infty } $ in $ \Gamma \setminus \ker \of{ \varphi } $ with
\[
\lim_{ n \to \infty }{ \frac{ \log \of{ \lambda_{ k } \of{ \zeta \of{ \gamma_{ n } } } } }{ \abs{ \varphi \of{ \gamma_{ n } } } } } = \inf_{ \varphi \of{ \gamma } \neq 0 }{ \frac{ \log \of{ \lambda_{ k } \of{ \zeta \of{ \gamma } } } }{ \abs{ \varphi \of{ \gamma } } } }
\]
has a subsequence (which we will also denote $ \of{ \gamma_{ n } }_{ n = 0 }^{ \infty } $) so that $ \of{ \frac{ \log \of{ \lambda_{ k } \of{ \zeta \of{ \gamma_{ n } } } } }{ \abs{ \varphi \of{ \gamma_{ n } } } } }_{ n = 0 }^{ \infty } $ is a sequence of pairwise distinct real numbers. Since $ \frac{ \log \of{ \lambda_{ k } \of{ \zeta \of{ \gamma } } } }{ \abs{ \varphi \of{ \gamma } } } $ is a function of the conjugacy class of $ \gamma $, this implies that $ \gamma_{ m } $ and $ \gamma_{ n } $ are not conjugate whenever $ m \neq n $, and so $ \lim_{ n \to \infty }{ \norm{ \gamma_{ n } }_{ \Sigma } } = \infty $. \qedhere

\end{proof}

\end{lemma+}

We are now ready to prove \autoref{cor:bounds}, which we restate below for the reader's convenience. The proof follows immediately from \autoref{cor:Fuchsian} as well as the above bounds in \autoref{lem:bounds}.

\bounds*

\begin{proof}

Let $ r^{ - } = s_{ k } \of{ \zeta } $ and $ r^{ + } = \max_{ \sigma \in \Sigma }{ \log \of{ \lambda_{ k } \of{ \zeta \of{ \sigma } } } } $. Note that \autoref{thm:symmetricAnosov} implies that $ 0 \in A^{ p }_{ k } \of{ \zeta } $, so we may assume that $ \varphi $ is not identically zero. In this case, \autoref{lem:bounds} gives
\[
\frac{ r^{ - } }{ \norm{ \varphi }_{ \infty } } \leq \inf_{ \varphi \of{ \gamma } \neq 0 }{ \frac{ \log \of{ \lambda_{ k } \of{ \zeta \of{ \gamma } } } }{ \abs{ \varphi \of{ \gamma } } } } \leq \frac{ r^{ + } }{ \norm{ \varphi }_{ \infty } } .
\]
If $ \norm{ \varphi }_{ \infty } < \frac{ p q r^{ - } }{ p + q } $, then
\[
\inf_{ \varphi \of{ \gamma } \neq 0 }{ \frac{ \log \of{ \lambda_{ k } \of{ \zeta \of{ \gamma } } } }{ \abs{ \varphi \of{ \gamma } } } } \geq \frac{ r^{ - } }{ \norm{ \varphi }_{ \infty } } > \frac{ r^{ - } \of{ p + q } }{ p q r^{ - } } = \frac{ p + q }{ p q } ,
\]
and so $ \varphi \in A^{ p }_{ k } \of{ \zeta } $ by \autoref{thm:symmetricAnosov}. Similarly, if $ \norm{ \varphi }_{ \infty } \geq \frac{ p q r^{ + } }{ p + q } $, then
\[
\inf_{ \varphi \of{ \gamma } \neq 0 }{ \frac{ \log \of{ \lambda_{ k } \of{ \zeta \of{ \gamma } } } }{ \abs{ \varphi \of{ \gamma } } } } \leq \frac{ r^{ + } }{ \norm{ \varphi }_{ \infty } } \leq \frac{ r^{ + } \of{ p + q } }{ p q r^{ + } } = \frac{ p + q }{ p q } ,
\]
and so $ \varphi \notin A^{ p }_{ k } \of{ \zeta } $ by \autoref{thm:symmetricAnosov}. \qedhere

\end{proof}

\section{Non-Anosov Limits of Anosov Reducible Suspensions} \label{sec:limits}

Consider an Anosov representation $ \eta \colon \Gamma \to \SL_{ 2 }^{ * } \of{ \K } $. In this section, we will investigate the boundaries of the convex domains $ A^{ p }_{ k } \of{ \iota_{ q } \circ \eta } $ in $ \hom \of{ \Gamma , \R } $. \autoref{cor:Fuchsian} implies that for a fixed $ p $ and $ q $, these domains are strictly nested, and so we are able to find continuous deformations of Borel Anosov representations of $ \Gamma $ along which the Anosov conditions fail one by one. This is the content of \autoref{cor:deformations}, which we restate below for the reader's convenience.

\deformations*

\begin{proof}

Let $ p = 1 $ and $ q = d - 1 $, and fix an Anosov representation $ \eta \colon \Gamma \to \SL_{ 2 }^{ * } \of{ \K } $ and a non-zero homomorphism $ \varphi \colon \Gamma \to \R $. Let
\[
a = \inf_{ \varphi \of{ \gamma } \neq 0 }{ \frac{ \log \of{ \lambda_{ 1 } \of{ \eta \of{ \gamma } } } }{ \abs{ \varphi \of{ \gamma } } } } .
\]
Consider for each $ 0 \leq t < \infty $ the reducible suspension $ \rho_{ t } = \rho_{ p , q } \of{ t \varphi , 1 , \iota_{ q } \circ \eta , \vect{ 0 } } $, so that $ t \mapsto \rho_{ t } $ is a continuous path in $ \hom \of{ \Gamma , \SL_{ d }^{ * } \of{ \K } } $. That $ \rho_{ 0 } $ is Borel Anosov follows immediately from \autoref{thm:symmetricAnosov} and the fact that
\[
\floor{ \frac{ q }{ 2 } } = \floor{ \frac{ 1 + q }{ 2 } } = \floor{ \frac{ d }{ 2 } } ,
\]
since $ q $ is even. On the other hand,
\[
\inf_{ t \varphi \of{ \gamma } \neq 0 }{ \frac{ \log \of{ \lambda_{ 1 } \of{ \eta \of{ \gamma } } } }{ \abs{ t \varphi \of{ \gamma } } } } = \frac{ 1 }{ t } \inf_{ \varphi \of{ \gamma } \neq 0 }{ \frac{ \log \of{ \lambda_{ 1 } \of{ \eta \of{ \gamma } } } }{ \abs{ \varphi \of{ \gamma } } } } = \frac{ a }{ t }
\]
for all $ 1 \leq t < \infty $. By \autoref{cor:Fuchsian}, when
\[
\frac{ a p q \of{ q - 2 j - 1 } }{ p + q } = \frac{ a \of{ d - 1 } \of{ d - 2 j - 2 } }{ d } \leq t < \frac{ a \of{ d - 1 } \of{ d - 2 j } }{ d } = \frac{ a p q \of{ q - 2 j + 1 } }{ p + q } ,
\]
$ \rho_{ t } $ is $ P_{ k } $-Anosov for all integers $ 1 \leq k \leq j $ but not $ P_{ k } $-Anosov for any integer $ j < k \leq \frac{ d }{ 2 } $. \qedhere

\end{proof}

\autoref{prop:eigenvalues} implies that if $ d $ is even, then no reducible suspension in $ \SL_{ d }^{ * } \of{ \K } $ can be Borel Anosov. Thus the conclusions of \autoref{cor:deformations} cannot be true if $ d $ is even. However, if we call a representation $ \Gamma \to \SL_{ d } \of{ \K } $ \emph{almost Borel Anosov} if it is $ P_{ k } $-Anosov for all integers $ 1 \leq k \leq \frac{ d - 2 }{ 2 } $, then there is an analogous result in this case which follows from an almost identical proof.

\begin{corollary+}

If $ \Gamma $ admits an Anosov representation in $ \SL_{ 2 }^{ * } \of{ \K } $ and $ d \geq 2 $ is even, then for all integers $ 1 \leq j \leq \frac{ d - 2 }{ 2 } $, there are continuous deformations of an almost Borel Anosov representation in $ \SL_{ d }^{ * } \of{ \K } $ which are $ P_{ k } $-Anosov for all integers $ 1 \leq k \leq j $ but not $ P_{ k } $-Anosov for any integer $ j < k \leq \frac{ d }{ 2 } $.

\end{corollary+}

For the remainder of this article, we now restrict our attention to the case in which $ \Gamma $ is a closed orientable surface group of genus $ g \geq 2 $ and $ \K = \R $. In this case, (the conjugacy class of the projectivization of) a discrete and faithful representation $ \eta \colon \Gamma \to \SL_{ 2 }^{ \pm } \of{ \R } $ (indeed, the image of any such representation lies in $ \SL_{ 2 } \of{ \R } $) yields a closed hyperbolic surface $ X $ and an explicit identification of $ \Gamma $ with the fundamental group $ \pi_{ 1 } \of{ X } $. In this section, we will explicitly construct points on the boundary of $ A^{ p }_{ k } \of{ \iota_{ q } \circ \eta } $ corresponding to particularly short non-separating geodesics on $ X $. Recall that we require $ 1 \leq k \leq \frac{ q }{ 2 } $ so that $ A^{ p }_{ k } \of{ \iota_{ q } \circ \eta } $ is non-empty.

Specifically, fix a simple non-separating element $ \alpha_{ 1 } \in \Gamma $. There are other simple non-separating elements $ \alpha_{ 2 } , \dotsc , \alpha_{ g } , \beta_{ 1 } , \dotsc , \beta_{ g } \in \Gamma $ so that
\[
\Gamma = \gen{ \alpha_{ 1 } , \dotsc , \alpha_{ g } , \beta_{ 1 } , \dotsc , \beta_{ g } \mid \comm{ \alpha_{ 1 } }{ \beta_{ 1 } } \dotsm \comm{ \alpha_{ g } }{ \beta_{ g } } } .
\]
Consider the homomorphism $ \varphi \colon \Gamma \to \R $ defined by the conditions
\begin{align*}
\varphi \of{ \alpha_{ 1 } } & = 1 & \varphi \of{ \alpha_{ 2 } } & = \dotsb = \varphi \of{ \alpha_{ g } } = \varphi \of{ \beta_{ 1 } } = \dotsb = \varphi \of{ \beta_{ g } } = 0 .
\end{align*}
In particular, note that $ \varphi \of{ \gamma } = \inner{ \class{ \beta_{ 1 } } }{ \class{ \gamma } } $ for all $ \gamma \in \Gamma $, where $ \inner{ \oparg }{ \oparg } $ is the algebraic intersection number. Note that here and in what follows, we will denote by $ \class{ \gamma } $ the closed geodesic on $ X $ represented by an element $ \gamma \in \Gamma $.

$ \R \varphi $ defines a line in $ \hom \of{ \Gamma , \R } $ through the origin, and so crosses $ \partial A^{ p }_{ k } \of{ \iota_{ q } \circ \eta } $ (which is the boundary of a non-empty bounded convex domain containing the origin) in precisely two points.

\begin{proposition+} \label{prop:boundary}

If $ \class{ \alpha_{ 1 } } $ is the shortest homotopically non-trivial geodesic arc from $ \class{ \beta_{ 1 } } $ to itself, then $ t \varphi \in \partial A^{ p }_{ k } \of{ \iota_{ q } \circ \eta } $ precisely when
\[
\abs{ t } = \frac{ p q \of{ q - 2 k + 1 } }{ p + q } \cdot \len_{ X } \of{ \class{ \alpha_{ 1 } } } .
\]

\begin{proof}

First, note that if $ \inf_{ \varphi \of{ \gamma } \neq 0 }{ \frac{ \len_{ X } \of{ \class{ \gamma } } }{ \abs{ \inner{ \class{ \beta_{ 1 } } }{ \class{ \gamma } } } } } = \len_{ X } \of{ \class{ \alpha_{ 1 } } } $, then \autoref{cor:Fuchsian} implies that $ t \varphi \in A^{ p }_{ k } \of{ \iota_{ q } \circ \eta } $ if and only if
\[
\frac{ \len_{ X } \of{ \class{ \alpha_{ 1 } } } }{ \abs{ t } } = \frac{ 1 }{ \abs{ t } } \inf_{ \varphi \of{ \gamma } \neq 0 }{ \frac{ \len_{ X } \of{ \class{ \gamma } } }{ \abs{ \inner{ \class{ \beta_{ 1 } } }{ \class{ \gamma } } } } } = \inf_{ t \varphi \of{ \gamma } \neq 0 }{ \frac{ \log \of{ \lambda_{ 1 } \of{ \eta \of{ \gamma } } } }{ \abs{ t \varphi \of{ \gamma } } } } > \frac{ p + q }{ p q \of{ q - 2 k + 1 } } .
\]
Since $ \frac{ \len_{ X } \of{ \class{ \alpha_{ 1 } } } }{ \abs{ t } } $ is a continuous and monotone function of $ t $, we may conclude that $ t \varphi \in \partial A^{ p }_{ k } \of{ \iota_{ q } \circ \eta } $ if and only if
\[
\abs{ t } = \frac{ p q \of{ q - 2 k + 1 } }{ p + q } \cdot \len_{ X } \of{ \class{ \alpha_{ 1 } } } .
\]
Thus it suffices to show $ \inf_{ \varphi \of{ \gamma } \neq 0 }{ \frac{ \len_{ X } \of{ \class{ \gamma } } }{ \abs{ \inner{ \class{ \beta_{ 1 } } }{ \class{ \gamma } } } } } = \len_{ X } \of{ \class{ \alpha_{ 1 } } } $. To that end, first note that
\[
\inf_{ \varphi \of{ \gamma } \neq 0 }{ \frac{ \len_{ X } \of{ \class{ \gamma } } }{ \abs{ \inner{ \class{ \beta_{ 1 } } }{ \class{ \gamma } } } } } \leq \frac{ \len_{ X } \of{ \class{ \alpha_{ 1 } } } }{ \abs{ \inner{ \class{ \beta_{ 1 } } }{ \class{ \alpha_{ 1 } } } } } = \len_{ X } \of{ \class{ \alpha_{ 1 } } } .
\]
On the other hand, fix $ \gamma \notin \ker \of{ \varphi } $. In particular, the geometric intersection number $ i \of{ \class{ \beta_{ 1 } } , \class{ \gamma } } \neq 0 $. Note that the points where $ \class{ \beta_{ 1 } } $ and $ \class{ \gamma } $ intersect divide $ \class{ \gamma } $ into $ i \of{ \class{ \beta_{ 1 } } , \class{ \gamma } } $ geodesic arcs from $ \class{ \beta_{ 1 } } $ to itself, which each have length at least $ \len_{ X } \of{ \class{ \alpha_{ 1 } } } $ by assumption. Thus
\[
\frac{ \len_{ X } \of{ \class{ \gamma } } }{ \abs{ \inner{ \class{ \beta_{ 1 } } }{ \class{ \gamma } } } } \geq \frac{ i \of{ \class{ \beta_{ 1 } } , \class{ \gamma } } \cdot \len_{ X } \of{ \class{ \alpha_{ 1 } } } }{ i \of{ \class{ \beta_{ 1 } } , \class{ \gamma } } } = \len_{ X } \of{ \class{ \alpha_{ 1 } } } .
\]
Since this holds for all $ \gamma \notin \ker \of{ \varphi } $, $ \inf_{ \varphi \of{ \gamma } \neq 0 }{ \frac{ \len_{ X } \of{ \class{ \gamma } } }{ \abs{ \inner{ \class{ \beta_{ 1 } } }{ \class{ \gamma } } } } } = \len_{ X } \of{ \class{ \alpha_{ 1 } } } $. \qedhere

\end{proof}

\end{proposition+}

We note that one may obtain a representation $ \eta \colon \Gamma \to \SL_{ 2 } \of{ \R } $ so that $ \alpha_{ 1 } $ satisfies the hypotheses of \autoref{prop:boundary} by pinching $ \class{ \alpha_{ 1 } } $ until it is very short.

\printbibliography

\end{document}